\documentclass[final]{siamart190516}
\usepackage{amssymb,amsmath,color}

\usepackage{amssymb}
\usepackage{fullpage}
\usepackage{enumitem}

\usepackage{xcolor, framed}
\newenvironment{myframedeq}[1][\linewidth]{\FrameSep=4pt\abovedisplayskip=0pt\belowdisplayskip=0pt
\framed\hsize=#1\leftskip=\dimexpr(\textwidth-#1)/2\relax}
{\endframed}

\newtheorem{remark}{Remark}[section]

\newcommand{\tcb}{\textcolor{blue}}
\newcommand{\tcr}{\textcolor{red}}

\newcommand{\R}{{\mathbb R}}
\newcommand{\N}{{\mathbb N}}
\DeclareMathOperator{\argmin}{argmin}
\newcommand{\interior}{{\rm int}\kern 0.06em}

\newcommand{\cC}{{\mathcal C}}
\newcommand{\cE}{{\mathcal E}}
\newcommand{\cH}{{\mathcal{H}}}
\newcommand{\cO}{{\mathcal O}}

\newcommand{\demi}{\frac{1}{2}}
\newcommand{\ie}{{\it i.e.}\,\,}

\def\<{\langle}
\def\>{\rangle}

\newcommand\prox{{\rm prox}}%

\newcommand{\norm}[1]{\left\|{#1}\right\|}
\newcommand{\pa}[1]{\left({#1}\right)}
\newcommand{\bpa}[1]{\big({#1}\big)}

\newcommand{\IGAHD}{\text{\rm{IGAHD}}}

\newcommand{\AVD}[1]{\text{\rm{AVD}$_{#1}$}}
\newcommand{\HBF}{\text{\rm{HBF}}}
\newcommand{\DINAVD}[1]{\text{\rm{DIN-AVD}$_{#1}$}}

\newcommand{\ISIHD}{\text{\rm{ISIHD}}}
\newcommand{\RAG}[1]{\text{\rm{RAG}$_{#1}$}}
\newcommand{\NAG}[1]{\text{\rm{NAG}$_{#1}$}}

\newcommand{\RAPG}[1]{\text{\rm{RAPG}$_{#1}$}}

\newcommand{\seq}[1]{\pa{#1}_{k \in \N}}
\DeclareMathOperator*{\wlim}{w-lim}
\newcommand{\qandq}{\quad \text{and} \quad}

\newcommand{\eqdef}{:=}

\usepackage{ulem}

\usepackage{subcaption}

\usepackage{pict2e}

 \if
{
\usepackage[pageref]{backref}
\renewcommand*{\backrefalt}[4]{%
\ifcase #1 %
(Not cited)%
\or
(Cited on p.~#2)%
\else
(Cited on pp.~#2)%
\fi
}

}
\fi

\newcommand{\TheTitle}{From the Ravine method to the Nesterov method and vice versa: a dynamical system perspective}

\title{{\TheTitle}}

\author{
  Hedy Attouch\thanks{IMAG, Universit\'e Montpellier, CNRS UMR 5149, Place Eug\`ene Bataillon, 34095 Montpellier Cedex 5, France. E-mail: \texttt{hedy.attouch@umontpellier.fr}.}
  \and
  Jalal Fadili\thanks{GREYC CNRS UMR 6072, Ecole Nationale Sup\'erieure d'Ing\'enieurs de Caen, 14050 Caen Cedex France. 
Email: \texttt{Jalal.Fadili@greyc.ensicaen.fr}.}
}

\begin{document}

\maketitle

\begin{abstract}
We revisit the Ravine method of Gelfand and Tsetlin from a dynamical system perspective, study its convergence properties, and highlight its similarities and differences with the Nesterov accelerated gradient method. The two methods are closely related. They can be deduced from each other by reversing the order of the extrapolation and gradient operations in their definitions.
They benefit from similar fast convergence of values and convergence of iterates for general convex objective functions. We will also establish the high resolution ODE of the Ravine and Nesterov methods, and reveal an additional geometric damping term driven by the Hessian for both methods. This will allow us to prove fast convergence towards zero of the gradients not only for the Ravine method but also for the Nesterov method for the first time. We also highlight connections to other algorithms stemming from more subtle discretization schemes, and finally describe a Ravine version of the proximal-gradient algorithms for general structured smooth + non-smooth convex optimization problems.
\end{abstract}



\section{Introduction}\label{sec:prel} 
In a real Hilbert space $\cH$, we revisit  the Ravine method of Gelfand and Tsetlin \cite{GT} from a dynamical system perspective, study its fast convergence properties, and compare it with the Nesterov accelerated gradient method \cite{Nest1,Nest2}, which we coin here \NAG{} for short. We first consider the case of smooth convex optimization 
\begin{equation}\label{basic-min-smooth}
\min \left\lbrace  f(x)  : \, x\in\cH  \right\rbrace,
\end{equation}
where  $f: \cH \rightarrow \R $  is a convex function of  class  ${\cC}^1$, whose gradient $\nabla f $   is Lipschitz continuous, and which satisfies $\argmin_{\cH}(f) \neq \emptyset$.
We will unveil the close links between the Ravine method and \NAG{} which are sometimes confused in the literature. Indeed, the two methods stem from different discretizations of similar continuous dynamics and can be deduced from each other by reversing the order of the extrapolation and gradient update operations in their definitions.
Thus, they benefit from similar fast convergence properties.
On the other hand, the high resolution ODE of the Ravine and Nesterov methods reveal an additional geometric Hessian-driven damping term. The Hessian damping, which is a special case of strong damping in PDE's, plays an important role in attenuating the oscillations. This paves the way to proving new results on fast convergence towards zero of the gradients for both methods. To achieve even better attenuation of the oscillations, we also highlight connections to other algorithms stemming from more subtle discretization schemes of the high resolution ODE. We finally examine the case of "smooth + nonsmooth" structured convex optimization problems, and introduce a first-order inertial proximal gradient algorithm which is based on the Ravine method.

\section{Damped inertial dynamical systems for fast optimization}
Damped inertial dynamics have a natural mechanical and physical interpretation. Asymptotically, they tend to stabilize the system at a minimizer of the global energy function. As such, they offer an intuitive way to develop fast optimization methods. 
Let us briefly describe the main damped inertial dynamics used in optimization, their mechanical interpretation, and how the Ravine method stands among them.
The Ravine method was a precursor of the accelerated gradient methods. It has long been ignored and, surprisingly enough, is at the forefront of current research. According to the notes in \cite{SuvritSra} :

"\textit{Ravine method worked well and sparked numerous heuristics for selecting its parameters and improving its behavior. However, its convergence was never proved. It inspired Polyak's heavy-ball method, which seems to have inspired Nesterov's accelerated gradient method}".

\subsection{Heavy Ball with Friction \eqref{ODE001}}
The heavy ball with friction method was introduced by Polyak in 1964 \cite{Polyak2,Polyak1}. It describes the movement of a material point of positive mass subjected to a driving force governed by the gradient of the function to be minimized and a viscous friction force. According to the fundamental equation of mechanics, and having normalized the mass equal to one, it is written as follows
\begin{equation}\tag{\HBF}\label{ODE001}
\ddot{x}(t) + \gamma \dot{x}(t) + \nabla f (x(t))=0,
\end{equation}
where $\gamma >0$ is a fixed viscous damping parameter. 
The \eqref{ODE001} method proves to be a useful tool for exploring the local minima of a smooth non-convex function \cite{AGR}.
For convex optimization, it is especially interesting in the strongly convex case, where an appropriate choice of the damping coefficient provides  linear convergence with an optimal rate.
Unfortunately, in the case of a general convex function, it only provides a sublinear rate of convergence of values of order $\cO \left(\frac{1}{t} \right)$. A decisive step to improve it, and to pass from the rate $\cO \left(\frac{1}{t} \right)$ to the faster rate $\cO \left(\frac{1}{t^2} \right)$,  has been accomplished by considering  algorithms associated with inertial dynamics with asymptotically vanishing damping coefficients.
This is the motivation behind the dynamic \eqref{eq:AVD} and associated algorithm \eqref{eq:NAG} described hereafter.

\subsection{Nesterov Accelerated Gradient method \eqref{eq:NAG}}
In recent years, an in-depth study was carried out linking the \NAG{} method to inertial dynamics with vanishing viscosity, see \cite{AC1,AC2,AC2R-EECT,ACPR,CEG1,CEG2,SBC}. Given $\alpha$ a  positive parameter, the following second-order ODE
\begin{equation}\tag{\AVD{\alpha}}\label{eq:AVD}
\ddot{x}(t) + \frac{\alpha}{t} \dot{x}(t) + \nabla f (x(t))=0,
\end{equation}
was introduced in \cite{SBC}. An appropriate temporal discretized version of this ODE with step size $s>0$ gives the scheme \eqref{eq:NAG} which reads
\begin{myframedeq}[0.6\linewidth]
\begin{equation}\tag{\NAG{\alpha}}\label{eq:NAG}
\quad\quad
\begin{cases}
y_k&= x_{k} + \bpa{1 - \frac{\alpha}{k}}( x_{k}  - x_{k-1}) \\
x_{k+1}&= y_k - s \nabla f(y_k).
\end{cases}
\end{equation}
\end{myframedeq}
The scheme \eqref{eq:NAG} performs a gradient step at $y_k$, which is an extrapolated point obtained from $x_k$ and the previous iterate $x_{k-1}$.

The method depends in an essential way on the tuning of the extrapolation parameter $\alpha_k$ which takes the form $\alpha_k = 1 - \frac{\alpha}{k}$ in the \eqref{eq:NAG} scheme. So $\alpha_k$ tends to one from below in a subtle controlled way. The  historical version of the accelerated gradient method of Nesterov corresponds to $\alpha=3$, with the asymptotic convergence rate
$f(x(t)) - \min_{\cH} f = \cO \left( \frac{1}{t^2} \right)$ for the continuous dynamic \eqref{eq:AVD}, and 
$f(x_k) - \min_{\cH} f = \cO \left( \frac{1}{k^2} \right)$  for the corresponding scheme \eqref{eq:NAG}. Taking $\alpha >3$ provides convergence  of the trajectories and the improved convergence rate,
with small $o$ instead of capital $\cO$ in the above convergence rates. These results are obtained by Lyapunov analysis \cite{ACPR,AP,CD}, as summarized below.
\begin{theorem}\label{thm:NAG}
Suppose that $f: \cH \to \R$ is a convex differentiable function such that $\nabla f$ is $L$-Lipschitz continuous, and $S = \argmin_{\cH}(f) \neq \emptyset$. Take $\alpha \geq 3$, and $s \in ]0,1/L]$. Let $\seq{x_k}$ be a sequence  generated by the \eqref{eq:NAG} algorithm.  Set $t_{k} =\frac{k-1}{\alpha-1}$, and define, for each integer $k\geq 1$
\[
E_k \eqdef t_k^2( f (x_k)-f(z) ) +\frac{1}{2s}\|x_{k-1}  - z + t_k (x_{k} - x_{k-1} )\|^2 .
\] 
Then, the sequence $\seq{E_k}$ is non-increasing, and as $ k\to +\infty$
\[
f(x_k)-\min_{\cH} f = \cO\left(\frac{1}{k^2}\right),  \quad  \|x_k - x_{k-1}\| =\cO\left(\frac{1}{k}\right).
\]
In addition, when $\alpha>3$, 
\begin{eqnarray*}
f(x_k)-\min_{\cH} f=o\left(\frac{1}{k^2}\right), \quad \|x_k - x_{k-1}\| = o\left(\frac{1}{k}\right) , \qandq \wlim x_k =x^\star \in S ,
\end{eqnarray*}
where $\wlim$ stands for the weak limit.
\end{theorem}

There has been an active literature devoted to study these questions, from various perspectives, which have given an in-depth understanding of the \eqref{eq:NAG} method; see for example \cite{AAD1,AC1,AC2,ACPR,ACR-subcrit,ADR,CD,Kim-F,LFP,May,MJ,Siegel,SDJS,SBC,VSBV}.

\subsection{Ravine method}\label{sec:Rav} 
The Ravine  method was introduced by Gelfand and  Tsetlin \cite{GT} in 1961. It mimics the flow of water in the mountains which first flows rapidly downhill through small, steep ravines and then flows along the main river in the valley. It also models the transmission of nerve impulses. It has been recently brought to the fore by Polyak \cite{Pol}. According to the above mechanical interpretation, the Ravine Accelerated Gradient method (coined \RAG{} for short) generates sequences $\seq{y_k}$ which satisfy
\begin{myframedeq}[0.6\linewidth]
\begin{equation}\tag{\RAG{\alpha}}\label{eq:RAG}
\begin{cases}
w_k=  y_k - s \nabla f(y_k)  \\
y_{k+1} = w_k + \bpa{1 -\frac{\alpha}{k+1}} \left( w_k - w_{k-1}\right).
\end{cases}
\end{equation}
\end{myframedeq}

Historically, the Ravine method was introduced with a fixed extrapolation coefficient. Taking the extrapolation coefficient equal to $\bpa{1 -\frac{\alpha}{k+1}}$ makes the Ravine method in accordance with \eqref{eq:NAG} and is crucial to obtain an accelerated method. A geometric view of \eqref{eq:RAG} is given in Figure~\ref{figRAG}. 
\begin{figure}[htbp]
\centering
\fbox{\begin{minipage}{12cm}

\setlength{\unitlength}{9cm}
\begin{picture}(1,0.7)(-0.3,0.04)

\put(0.364,0.647){$\bullet$}
\tcr{
\put(0.37,0.648){\vector(-1,-4){0.022}}
}

\put(0.34,0.55){$\bullet$}

\put(0.278,0.605){$\bullet$}
\tcr{
\put(0.287,0.61){\vector(-1,-3){0.031}}
}

\put(0.25,0.505){$\bullet$}

\tcr{
 \put(0.341,0.555){\vector(-2,-1){0.19}}
 }
\put(0.025,0.205){\line(1,1){0.178}}
\put(0.145,0.26){\line(1,1){0.1}}
\put(-0.1,0.15){\line(1,1){0.13}}
\put(0.07,0.34){\line(1,1){0.05}}

\put(-0.12,0.22){\line(1,1){0.08}}
\put(-0.025,0.33){\line(1,1){0.07}}

\put(0.14,0.45){$\bullet$}

\tcb{
\put(0.366,0.669){$y_{k-1}$}
\put(0.28,0.633){$y_{k}$}
\put(0.365,0.56){$w_{k-1} =   y_{k-1} -s\nabla f (y_{k-1}) $}
\put(0.27,0.5){$w_{k} =  y_{k} -s \nabla f (y_{k}) $}
\put(0.164,0.445){$y_{k+1}=  w_k + \pa{1 - \frac{\alpha}{k+1}} \left( w_k - w_{k-1}\right)$}
\put(-0.1,0.3){$S = \argmin_{\cH}(f)$}
}

\qbezier(-0.1,0.15)(0.6,0.43)(-0.1,0.4)
\qbezier(0.3,0.15)(1.0,0.7)(-0.1,0.69)
\qbezier(0.2,0.15)(0.9,0.637)(-0.1,0.639)

\end{picture}
\end{minipage}}
\caption{A geometrical illustration of \eqref{eq:RAG}.}
\label{figRAG}
\end{figure}

\subsection{From Ravine to Nesterov and vice versa}\label{sec:RavNest}
The Ravine method has been ignored for a long time, and sometimes confused with \eqref{eq:NAG} in the literature. Indeed, the two methods can be deduced from each other by reversing the order of the extrapolation and gradient operations. Even more confusing, they come within the same equations. Specifically, the variable $y_k$ which enters the definition of \eqref{eq:NAG} follows the \eqref{eq:RAG} algorithm. Despite the elementary proof of this result, we state it as a theorem, because of its importance. 

\begin{theorem}\label{Nest_Ravine}
{~}\\\vspace*{-0.5cm}
\begin{enumerate}[label=(\roman*)]
\item Let $\seq{x_k}$ be a sequence generated by the algorithm \eqref{eq:NAG}. Let $\seq{y_k}$ be the associated sequence given by $y_k=  x_{k} + \left(1 -\frac{\alpha}{k}\right) (x_{k}  - x_{k-1})$.
Then, $\seq{y_k}$ follows the algorithm \eqref{eq:RAG}.

\item Conversely, if $\seq{y_k}$ is a sequence generated by \eqref{eq:RAG}, then the sequence $\seq{x_k}$ defined by 
$
x_{k+1} =  y_k - s \nabla f (y_k)
$
obeys the algorithm \eqref{eq:NAG}.
\end{enumerate}
\end{theorem} 
\begin{proof}
\begin{enumerate}[label=(\roman*)]
\item Suppose that the iterates $\seq{x_k}$ follow \eqref{eq:NAG}. According to the definition of $y_k$
\begin{eqnarray*}
y_{k+1}
&=&  x_{k+1} + \pa{1 -\frac{\alpha}{k+1}} ( x_{k+1}  - x_{k})\\
&=&  y_k- s\nabla f (y_k) + \pa{1 -\frac{\alpha}{k+1}} \Big( y_k- s\nabla f (y_k)  - ( y_{k-1}- s\nabla f (y_{k-1}) )\Big).
\end{eqnarray*}
Set $w_k = y_k - s\nabla f (y_k)$ (which is nothing but $x_{k+1}$). We obtain that the sequence $\seq{y_k}$ obeys \eqref{eq:RAG}.
%
\item Conversely, from the definition of $y_{k}$ and $w_k$ in \eqref{eq:RAG}, we have 
\begin{eqnarray*}
y_{k+1} &=& y_k- s\nabla f (y_k) + \pa{1 - \frac{\alpha}{k+1}} \Big(y_k- s\nabla f (y_k) -( y_{k-1}- s\nabla f (y_{k-1}))\Big).
\end{eqnarray*}
Setting $x_{k+1}  \eqdef  y_k- s \nabla f (y_k)$ as devised, we deduce that 
\begin{eqnarray*}
y_{k+1} &= & x_{k+1} + \pa{1 - \frac{\alpha}{k+1}} \left(x_{k+1} - x_{k}\right).
\end{eqnarray*}
Equivalently
\begin{eqnarray*}
y_{k} &= & x_{k} + \left(1 -\frac{\alpha}{k}\right) \left(x_{k} - x_{k-1}\right).
\end{eqnarray*}
Putting together the above relation with the definition of $x_{k+1}$, we obtain that $\seq{x_k}$
follows \eqref{eq:NAG}.
\end{enumerate}
This completes the proof. 
\end{proof}

\begin{remark}{
The order of the two operations, gradient and extrapolation is reversed in the two algorithms. In \eqref{eq:NAG} first the extrapolation operation is performed, followed by a gradient step. In the Ravine method \eqref{eq:RAG}, it is the reverse order. Although different, this is reminiscent of the approach followed in \cite{APR-backward-forward} which makes it possible to switch from forward-backward algorithms to backward-forward algorithms.}
\end{remark}

Equipped with Theorem~\ref{Nest_Ravine}, we now transfer convergence properties known for \eqref{eq:NAG} to \eqref{eq:RAG}. In particular, we will show that \eqref{eq:NAG} and \eqref{eq:RAG} share the same asymptotic convergence rates.  

\begin{theorem} \label{Rav-conv}
Let $f : \cH \rightarrow \mathbb R$ be a $\cC^1$ convex function whose gradient is $L$-Lipschitz continuous, and  $S= \argmin_{\cH}(f) \neq \emptyset$. Let $\seq{y_k} $ be the sequence generated by \eqref{eq:RAG} with $\alpha \geq 3$ and $sL \leq 1$. Then, the following properties hold: 
\begin{enumerate}[label=(\roman*)]
\item $f (y_k)-\min_{\cH} f =  \cO \left(\displaystyle{\frac{1}{k^2} }\right)$  as $k \to +\infty$.

If, in addition, $\alpha >3$, then
\item $f (y_k)-\min_{\cH} f =  o \left(\displaystyle{\frac{1}{k^2} }\right)$  as $k \to +\infty$.

Let $\seq{x_k}$ be the associated trajectory generated by \eqref{eq:NAG}, \ie $x_{k+1} =  y_k - s \nabla f (y_k)$. Then
\item $\wlim y_k = \wlim x_k \in S$.
\end{enumerate}
\end{theorem}

\begin{proof}
\begin{enumerate}[label=(\roman*)]
\item According to Theorem~\ref{Nest_Ravine}, the sequence $\seq{x_k}$ defined by 
\begin{equation}\label{def:xk}
x_{k+1} =  y_k- s \nabla f (y_k)
\end{equation} 
follows \eqref{eq:NAG}. Let us take advantage of the convergence properties of \eqref{eq:NAG}, as described in Theorem~\ref{thm:NAG}. We thus have
\begin{equation}\label{NAGconvrate}
f(x_k)-\min_{\cH} f = \cO
\left(\frac{1}{k^2}\right),  \quad  \|x_k - x_{k-1}\| =\cO\left(\frac{1}{k}\right).
\end{equation} 
According to \eqref{def:xk}, we have $-\frac{1}{s} (x_{k+1} -  y_k ) =  \nabla f (y_k)$. Using successively the convex subdifferential inequality, the Cauchy-Schwarz inequality, and the triangle inequality, we obtain
\begin{eqnarray}
f(y_k) -\min_{\cH} f &\leq& f(x_k) -\min_{\cH} f + \frac{1}{s} \left\langle x_{k+1} -  y_k, x_k - y_k   \right\rangle \nonumber \\
&\leq& f(x_k)-\min_{\cH} f  + \frac{1}{s}  \|x_{k+1} -  y_k\|   \|y_{k} -  x_{k}\|, \nonumber \\
&\leq& f(x_k)-\min_{\cH} f  + \frac{1}{s} \left( \|x_{k+1} -  x_k\| + \|y_{k} -  x_{k}\| \right)  \|y_{k} -  x_{k}\|. \label{NAG2}
\end{eqnarray}
Using Theorem~\ref{Nest_Ravine} on the equivalence between \eqref{eq:RAG} and \eqref{eq:NAG}, we have 
\begin{eqnarray*}
y_{k} &= & x_{k} + \left(1 -\frac{\alpha}{k}\right) \left(x_{k} - x_{k-1}\right).
\end{eqnarray*}
Therefore
\begin{eqnarray}\label{NAG22}
\|y_{k}- x_{k} \| \leq \|x_{k} - x_{k-1}\|.
\end{eqnarray}
Combining \eqref{NAG2} with \eqref{NAG22} we obtain
\begin{equation}\label{NAG222}
f(y_k) -\min_{\cH} f \leq f(x_k)-\min_{\cH} f + \frac{1}{s} \left( \|x_{k+1} -  x_k\| + \|x_{k} - x_{k-1}\| \right)  \|x_{k} - x_{k-1}\|.
\end{equation}
According to \eqref{NAGconvrate} we conclude. 
\item Similar arguments when $\alpha >3$ yield the $o\pa{\frac{1}{k^2}}$ rate.

\item Since $\| y_k -x_k \| \leq \|x_{k} -  x_{k-1}\|$, and $  \|x_k - x_{k-1}\| =\cO
\left(\frac{1}{k}\right) \to 0$, we have $\| y_k -x_k \| \to 0$, \ie $y_k -x_k $ converges strongly to zero. Combining this with the fact that $\wlim x_k = x^\star \in S$, see Theorem~\ref{thm:NAG}, it follows that the sequence $\seq{y_k}$ converges weakly to the same limit $x^\star$.
\end{enumerate}
\end{proof}

\subsection{Inertial dynamics with Hessian driven damping}\label{sec:Hessian}
It the light of recent work in the study of the acceleration of first order algorithms through the lens of dynamical systems, we will show that the high resolution ODE's of \eqref{eq:RAG} and \eqref{eq:NAG} contain an additional Hessian-driven geometric damping term. The underlying dynamic, called \eqref{eq:DINAVD} is the subject of this section.

The Hessian driven damping plays an important role in various domains. As such, it can be introduced from various perspectives:
geometric (also called strong) damping of oscillating systems, which is a central theme in PDE's and mechanics, high resolution ODE of the Ravine method (this will be analyzed in section \ref{sec:dyn-perp}), and regularization of the Newton method.
In addition to the viscous damping already present in \eqref{eq:AVD}, taking into account the geometric damping which is driven by the Hessian of the function to be minimized makes it possible to improve the performance of these methods by attenuating their oscillations (see Figure~\ref{fig2D}). In the rest of this section, we follow the lines of \cite{ACFR}, see also \cite{SDJS} for a special case and motivation. 
When $f$ is twice continuously differentiable, the dynamic is  written as follows
\begin{myframedeq}[0.7\linewidth]
\begin{equation}\tag{\DINAVD{\alpha,\beta,b}}\label{eq:DINAVD}
\ddot{x}(t) + \frac{\alpha}{t} \dot{x}(t) + \beta \nabla^2 f (x(t))\dot{x}(t) +  b(t)\nabla f (x(t))=0,
\end{equation}
\end{myframedeq}
\noindent
where $b(t)$ takes into account the temporal scaling effect.
The prefix DIN, which stands shortly for Dynamical Inertial Newton, refers to the interpretation of this dynamic as a regularized continuous Newton method, see \cite{ASv,AL1,AL2}.
At first glance, the presence of the Hessian may seem to entail numerical difficulties. Fortunately, this is not the case as the Hessian intervenes  in the form $\nabla^2  f (x(t)) \dot{x} (t)$, which is nothing but the derivative with respect to time of the function $t \mapsto \nabla  f (x(t))$. 

\medskip 

The temporal discretization of the dynamic \eqref{eq:DINAVD} with $b(t)=1+\beta/t$, provides the first-order algorithm proposed in \cite{ACFR}
\begin{myframedeq}[\linewidth]
\begin{equation}\tag{\IGAHD}\label{eq:IGAHD}
\begin{cases}
y_k=   x_{k} + \left(1- \frac{\alpha}{k}\right) ( x_{k}  - x_{k-1}) -
\beta \sqrt{s}   \left( \nabla f (x_k)  - \nabla f (x_{k-1}) \right) -  \frac{\beta \sqrt{s}}{k}\nabla f( x_{k-1})\\
x_{k+1} = y_k - s \nabla f (y_k),
\end{cases}
\end{equation}
\end{myframedeq}
\noindent
where \eqref{eq:IGAHD} stands for Inertial Gradient Algorithm with Hessian driven Damping.
By comparison with \eqref{eq:NAG}, \eqref{eq:IGAHD} has a correction term which contains the difference of the gradients at two consecutive steps.
While preserving the convergence properties of \eqref{eq:NAG}, \eqref{eq:IGAHD} provides fast convergence to zero of the gradients, and  reduces the oscillatory aspects. This is made precise in the following theorem, see \cite{ACFR,ACFR-Optimisation}.
\begin{theorem}\label{thm:IGAHD}
 Let  $f: \cH \to \R$ be  a $\cC^1$ convex function whose gradient is $L$-Lipschitz continuous, and $S = \argmin_{\cH}(f)  \neq \emptyset$. 
Suppose that $\alpha \geq 3$, $0 < \beta <  2 \sqrt{s}$, $sL \leq 1$. Let $\seq{x_k}$ be a sequence generated by \eqref{eq:IGAHD}. Then, the following holds:
\begin{enumerate}[label=(\roman*)]
\item $f(x_k)-\min_{\cH} f = \cO 
\left(\dfrac{1}{k^2}\right)$,  and  $\|x_k - x_{k-1}\| =   \cO \left(\dfrac{1}{k}\right)$ as $k\to +\infty$;
\item $\sum_k  k^2 \| \nabla f (y_k) \|^2 < +\infty$ and $\sum_k  k^2 \| \nabla f (x_k) \|^2 < +\infty$.\\\\
In addition, when $\alpha>3$, 
\item $f(x_k)-\min_{\cH} f=o\left(\dfrac{1}{k^2}\right)$ and $\|x_k - x_{k-1}\| =   o\left(\dfrac{1}{k}\right)$;
\item $\wlim x_k \in S$. 
\end{enumerate}
\end{theorem}
\begin{figure}
\begin{center}
\includegraphics[width=0.8\textwidth]{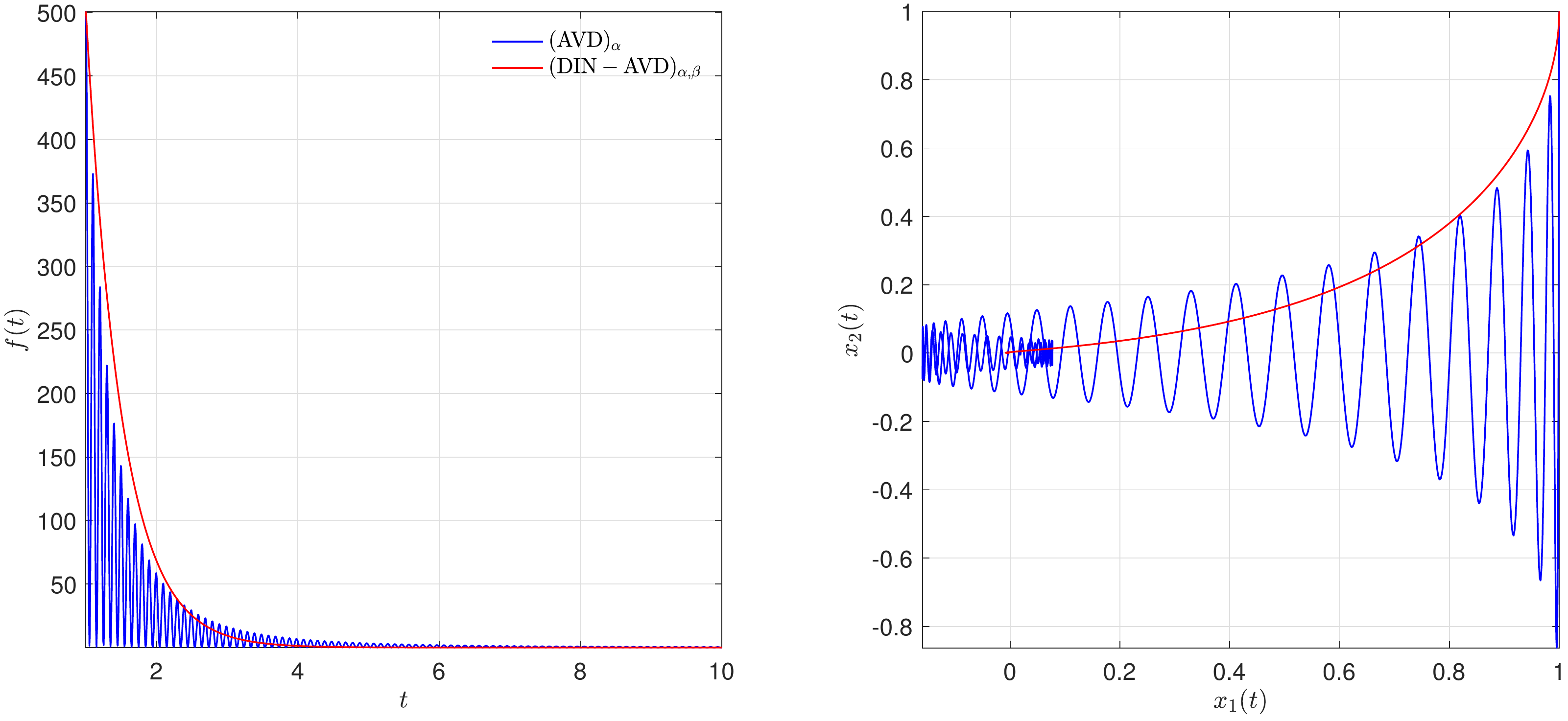}
\end{center}
\caption{Evolution of the objective (left) and trajectories (right) for \eqref{eq:AVD} ($\alpha=3.1)$ and \eqref{eq:DINAVD} ($\alpha=3.1,\beta=1$) on an ill-conditioned quadratic problem in $\R^2$.}
\label{fig2D}
\end{figure}

A number of other recent papers have contributed to this subject or closely related ones, see \cite{AAt1,AAt2,BCL,Kim,LJ}.

\medskip

Let us mention another important advantage of \eqref{eq:DINAVD}, which confirms its natural connection with first-order methods. When $\beta >0$, the presence of the Hessian driven damping in the dynamics \eqref{eq:DINAVD} allows to formulate \eqref{eq:DINAVD} as an equivalent first-order system both in time and in space, without explicit evaluation of the Hessian. This makes it possible to extend the existence of trajectories and the convergence results to the non-smooth case  $f \in \Gamma_0(\cH)$  (the class of proper, lower semicontinuous and convex functions on $\cH$), by simply replacing the gradient of $f$ by the subdifferential $\partial f$. This approach was initiated in \cite{AABR} and  \cite{APR,APR2}, and used in the perturbed case in \cite{AFV}.
From a mechanical perspective, non-smooth $f$ permits to model non-elastic shocks in unilateral mechanics, see \cite{AMR}.

\medskip

In \eqref{eq:DINAVD}, the Hessian appears explicitly. A closely related ODE is obtained by considering an approach where the Hessian driven damping appears in an implicit form. This was initiated in \cite{ALP}, see also \cite{MJ} for a related autonomous system  in the case of a strongly convex function $f$. This ODE, coined \eqref{eq:odetwo} for Inertial System with Implicit Hessian Damping, takes the form 
\begin{equation}\tag{\ISIHD}\label{eq:odetwo}
\ddot{x}(t)+\frac{\alpha}{t} \dot{x}(t)+\nabla f\bpa{x(t)+\beta(t)\dot{x}(t)}=0 ,
\end{equation}
where $\alpha \ge 3$ and $\beta(t)=\gamma+\frac{\beta}{t}$, $\gamma, \; \beta \geq 0$. The rationale justifying the use of the term ``implicit" comes from the observation that by a Taylor expansion (as $t \to +\infty$ we have $\dot{x}(t) \to 0$ which justifies using Taylor expansion), one has
\[
\nabla f\bpa{x(t)+\beta(t)\dot{x}(t)}\approx \nabla f (x(t)) + \beta(t)\nabla^2 f(x(t))\dot{x}(t) ,
\]
hence making the Hessian damping appear indirectly in \eqref{eq:odetwo}. As for \eqref{eq:DINAVD}, this ODE was found to have a smoothing effect on the oscillations.

\section{The dynamical system perspective of \eqref{eq:NAG}}

This sections reviews the close ties between the \eqref{eq:NAG} algorithm and the associated \eqref{eq:AVD} system. They will serve as a basis for exploring similar questions for \eqref{eq:RAG}. In doing so, we highlight general methods and tools that allow moving from continuous dynamics to algorithms via temporal discretization, and vice versa.

\subsection{From continuous dynamics to algorithms and vice versa}\label{continuous-discrete}
A general and successful recipe to pass from continuous gradient dynamics to gradient algorithms is to follow the following two-step procedure:

\begin{enumerate}[label=\roman*)]
\item First consider the implicit discretization of the continuous dynamic, and so obtain a proximal algorithm.
It is a well known fact that the implicit discretization usually preserves the asymptotic convergence properties of the continuous dynamic. This fact has been well documented for first-order evolution dynamics associated with convex optimization problems \cite{Pey_Sor}, and explains the importance of the proximal algorithm. This type of property is also directly linked with the exponential formula and the Trotter-Lie-Kato formula for the generation of contraction semigroups generated by maximally monotone operators. However, the proximal operator $\prox_{s f} \eqdef (I+s \nabla f)^{-1}$, $s > 0$, may not be easy to compute, which justifies the next step.

\item In the so obtained proximal algorithm, replace the proximal step associated with the operator $\prox_{s f}$ by a gradient step associated with the operator $I - s \nabla f$. By taking  $s$ sufficiently small (typically less than or equal to the inverse of the Lipschitz constant of the gradient of the function which is to be minimized), one can expect to preserve the convergence properties. 
\end{enumerate}
 
A major advantage of this procedure is that the proximal and gradient steps have a similar structure and are therefore likely to be combined in proximal gradient algorithms for structured optimization. 

\subsection{Passing from \eqref{eq:AVD} to \eqref{eq:NAG} by temporal discretization}\label{sec:implicit_discrete}
Let us illustrate the above procedure in the case of the \eqref{eq:AVD} dynamic.
Implicit time discretization of \eqref{eq:AVD}, with step size $h>0$, gives
\[
\frac{ x_{k+1} - 2x_{k}+ x_{k-1} }{h^2} +   \frac{\alpha}{kh} \frac{x_{k} - x_{k-1}}{h} + \nabla f( x_{k+1}) = 0.
\]
After multiplication by $s=h^2$, we obtain
\begin{equation}
(x_{k+1} -x_{k})- (x_{k} -x_{k-1}) + \frac{\alpha}{k}(x_{k} - x_{k-1}) + s \nabla f( x_{k+1})=0. \label{basic-eq-1}
\end{equation}
Equivalently
\begin{equation}
x_{k+1}  + s \nabla f( x_{k+1})= x_{k}+  \left( 1- \frac{\alpha}{k}\right)(x_{k} -x_{k-1}) , \label{basic-eq-2}
\end{equation}
which gives
\begin{equation}
x_{k+1} = \prox_{s f} \left(x_{k} +  \left( 1- \frac{\alpha}{k}\right)(x_{k} -x_{k-1}) \right). \label{basic-eq-3}
\end{equation}
So, we obtain the inertial proximal algorithm 
\begin{eqnarray*}
\begin{cases}
y_k=  \displaystyle{x_{k} + \pa{1 - \frac{\alpha}{k}}( x_{k}  - x_{k-1})} \\
x_{k+1} = \prox_{s f} \left( y_{k} \right).
\end{cases}
\end{eqnarray*}
This algorithm was initiated by G\"uler in \cite{Guler1,Guler2}. It is a key ingredient of the FISTA method \cite{BT}. Replacing the proximal step by a gradient step, we obtain the \eqref{eq:NAG} method.
%
\begin{remark}{
One may wonder if a full implicit discretization, which also involves the damping term, leads to the same algorithm. So consider
\[
\frac{ x_{k+1} - 2x_{k}+ x_{k-1} }{h^2} +   \frac{\alpha}{kh} \frac{x_{k+1} - x_{k}}{h} + \nabla f( x_{k+1}) = 0.
\]
Similar calculation as above gives the inertial proximal algorithm
\begin{eqnarray*}
\begin{cases}
y_k		=  {x_{k} + \frac{1}{1+ \frac{\alpha}{k}}( x_{k}  - x_{k-1})} \\
x_{k+1} = \prox_{\frac{s}{1+ \frac{\alpha}{k}} f} \left( y_{k} \right).
\end{cases}
\end{eqnarray*}
The gradient version of the above algorithm takes the form
\begin{eqnarray*}
\begin{cases}
y_k		=  \displaystyle{x_{k} + \frac{k}{k+\alpha}( x_{k}  - x_{k-1})} \\
x_{k+1} = y_k -\frac{sk}{k+\alpha}\nabla f \left( y_{k} \right).
\end{cases}
\end{eqnarray*}
The above form of the extrapolation coefficient, namely $\frac{k}{k+\alpha}$ is often used by practitioners, though they maintain the step size equal to $s$ rather than $\frac{sk}{k+\alpha}$, which is obviously asymptotically equal to $s$. This leads to results similar to those with the extrapolation coefficient $1 - \frac{\alpha}{k}$ which  is clearly asymptotically equivalent. This is well documented in \cite{AC2}, where the case of  general damping and extrapolation coefficients is considered.}
\end{remark}

\subsection{Passing from \eqref{eq:NAG} to \eqref{eq:AVD}}\label{NtoAVD}
We use here a standard limiting argument to pass from \eqref{eq:NAG} to \eqref{eq:AVD} (see also \cite{SBC}). First write \eqref{eq:NAG} equivalently as 
\[
x_{k+1} = x_{k} + \pa{1 - \frac{\alpha}{k}}( x_{k}  - x_{k-1}) - s \nabla f \Big(x_{k} + \pa{1 - \frac{\alpha}{k}}( x_{k}  - x_{k-1})\Big).
\]
Thus, with $s=h^2$, this is  also equivalent to
\begin{equation}\label{basic-NAG}
\frac{ x_{k+1} - 2x_{k}+ x_{k-1} }{h^2} +   \frac{\alpha}{kh} \frac{x_{k} - x_{k-1}}{h} + \nabla f(y_k)= 0.
\end{equation}
For each $k\in\N$, set $t_k=kh$, and we use the ansatz $x_k = X(t_k)$ for some smooth curve $t \mapsto X(t)$ defined for $t\geq t_0 >0$. Performing a Taylor expansion in powers of $h$, when $h$ is close  to zero, of the different quantities involved in \eqref{eq:NAG}, we obtain
\begin{eqnarray}
x_{k+1} &=& X(t_{k+1})= X(t_k) + h \dot{X}(t_k)+ \demi h^2 \ddot{X}(t_k) + \cO (h^3) \label{taylor1}\\
x_{k-1} &=& X(t_{k-1})= X(t_k) - h \dot{X}(t_k)+ \demi h^2 \ddot{X}(t_k) + \cO (h^3). \label{taylor2}
\end{eqnarray}
By adding \eqref{taylor1} and \eqref{taylor2},  we obtain 
\[
\frac{ x_{k+1} - 2x_{k}+ x_{k-1} }{h^2} =   \ddot{X}(t_k) + \cO (h).
\]
Moreover, \eqref{taylor2} gives
\[
\frac{x_{k} - x_{k-1}}{h} =   \dot{X}(t_k) + \cO (h).
\]
According to the $L$-Lipschitz continuity property of $\nabla f$, the definition of $y_k$ and $1-\alpha/k \leq 1$, we have
\begin{eqnarray*}
\| \nabla f \left( y_{k} \right) -\nabla f \left( x_{k} \right)\|
&\leq& L \| y_k -x_k  \|
\\
&\leq& L \|x_k -  x_{k-1}\|.
\end{eqnarray*}
Therefore
\begin{equation}\label{taylorgrad}
\nabla f \left( y_{k}\right) = \nabla f \left( X(t_k) \right) + \cO (h).
\end{equation}
Plugging \eqref{taylor1}, \eqref{taylor1} and \eqref{taylorgrad} into \eqref{basic-NAG}, we obtain
\begin{equation}\label{basic-NAG-2}
 \ddot{X}(t_k) +   \frac{\alpha}{t_k} \dot{X}(t_k) + \nabla f \left( X(t_k)\right) + \cO(h)= 0.
\end{equation}
When $h$ is small, we can neglect the $\cO(h)$ term which, at the limit, gives that $X(\cdot)$ follows the ODE \eqref{eq:AVD}.
%
 
\subsection{High resolution ODE of \eqref{eq:NAG}}
The high resolution  method is  extensively used in fluid mechanics, where physical phenomena occur at multiple scales, see for example \cite{Pedlosky} for a comprehensive presentation of geophysical fluid dynamics. The idea in our context is not to let $h \to 0$, but to take into account the terms of order $h=\sqrt{s}$ in the asymptotic expansions, and to discard the terms of order $h^2=s$ and higher. Moreover, to make the Hessian appear explicitly (see also the discussion on the system \eqref{eq:odetwo} above), we will have to refine the Taylor expansion \eqref{taylorgrad}. By doing so for \eqref{eq:NAG}, we now show that a Hessian-driven damping term appears in the associated continuous inertial ODE. This is a distinctly new feature and we are not aware of any such a result for \eqref{eq:NAG}.
\begin{theorem}\label{highresolutionnag}
Assume that $f$ is $\cC^2$. The high resolution ODE with temporal step size $\sqrt{s}$ of \eqref{eq:NAG} 
%
gives the inertial dynamic with Hessian driven damping
\begin{equation}\label{highresODENAG}
\ddot{X}(t) + \sqrt{s}\nabla^2 f (X(t))\dot{X}(t) 
+ \frac{\alpha}{t}\dot{X}(t) + \left(1+  \frac{\alpha \sqrt{s}}{2t} \right)  \nabla f (X(t)) =0 .
\end{equation}
\end{theorem}

\begin{proof}
Recall the  equivalent formulations of \eqref{eq:NAG} in \eqref{basic-NAG} with $s=h^2$.
For each $k\in\N$, set $t_k \eqdef h(k+c)$ for a real parameter $c$ to be adjusted later, and use the ansatz that $x_k = X(t_k)$ for some  smooth curve  $t \mapsto X(t)$ defined for $t\geq t_0 >0$. Performing a Taylor expansion in powers of $h$, when $h$ is close  to zero, of the different quantities involved in \eqref{eq:NAG}, we obtain
\begin{eqnarray}
x_{k+1} &=& X(t_{k+1})= X(t_k) + h \dot{X}(t_k)+ \demi h^2 \ddot{X}(t_k) + + \frac{1}{6}h^3 \dddot{X}(t_k)+ 
\mathcal O (h^4) \label{taylor1-b}\\
x_{k-1} &=& X(t_{k-1})= X(t_k) - h \dot{X}(t_k)+ \demi h^2 \ddot{X}(t_k) - \frac{1}{6}h^3 \dddot{X}(t_k)+ 
\mathcal O (h^4). \label{taylor2-b}
\end{eqnarray}
By adding \eqref{taylor1-b} and \eqref{taylor2-b},  we obtain 
\[
\frac{ x_{k+1} - 2x_{k}+ x_{k-1} }{h^2} =   \ddot{X}(t_k) + \mathcal O (h^2).
\]
Moreover, \eqref{taylor2-b} gives
\[
\frac{x_{k} - x_{k-1}}{h} =   \dot{X}(t_k)- \demi h \ddot{X}(t_k)  + \mathcal O (h^2).
\]
We also have
\begin{eqnarray*}
\nabla f(y_k) = \nabla f \pa{x_{k} + \pa{1 - \frac{\alpha}{k}}( x_{k}  - x_{k-1})}
&=& \nabla f \pa{x_{k} + h\pa{1 - \frac{\alpha}{k}}\frac{x_{k} - x_{k-1}}{h}} \\ 
&=& \nabla f \pa{X(t_k)  + h\pa{1 - \frac{\alpha}{k}}\pa{\dot{X}(t_k)- \demi h \ddot{X}(t_k)  + \mathcal O (h^2)}}\\ 
&=& \nabla f \pa{X(t_k)  + h\pa{1 - \frac{\alpha}{k}} \dot{X}(t_k)  + \mathcal O (h^2)}\\
&=& \nabla f (X(t_k))   + h\pa{1 - \frac{\alpha}{k}} \nabla^2 f (X(t_k))\dot{X}(t_k)  + \mathcal O (h^2) . 
\end{eqnarray*}
Putting this with \eqref{taylor1-b} and \eqref{taylor2-b} into \eqref{basic-NAG}, we obtain
\begin{equation}\label{basic-NAG-2-b}
\ddot{X}(t_k) +   \frac{\alpha}{kh} \left(\dot{X}(t_k) - \demi h \ddot{X}(t_k) \right)+  \nabla f (X(t_k))   + h\pa{1 - \frac{\alpha}{k}} \nabla^2 f (X(t_k))\dot{X}(t_k)  + \mathcal O (h^2) = 0.
\end{equation}
Equivalently,
\begin{equation}\label{basic-NAG-2-c}
\pa{1-\frac{\alpha}{2k}} \ddot{X}(t_k) + \frac{\alpha}{kh} \dot{X}(t_k) +  \nabla f (X(t_k))   + h\pa{1 - \frac{\alpha}{k}} \nabla^2 f (X(t_k))\dot{X}(t_k)  + \mathcal O (h^2) = 0.
\end{equation}
Dividing by $\pa{1-\frac{\alpha}{2k}}$ gives 
\begin{equation*}
\ddot{X}(t_k) + \frac{\alpha}{h(k-\frac{\alpha}{2} )} \dot{X}(t_k) + \left(1+  \frac{\alpha h}{2h(k-\frac{\alpha}{2} )} \right)  \nabla f (X(t_k))   + h\left(1 - \frac{\frac{\alpha}{2}}{k -\frac{\alpha}{2}} \right) \nabla^2 f (X(t_k))\dot{X}(t_k)  + \mathcal O (h^2) = 0.
\end{equation*}
Set $c = -\frac{\alpha}{2}$ and thus $t_k \eqdef h(k-\frac{\alpha}{2})$. We obtain
\begin{equation*}
\ddot{X}(t_k) +   \frac{\alpha}{t_k} \dot{X}(t_k) + \left(1+  \frac{\alpha h}{2t_k} \right)  \nabla f (X(t_k))   + h\left(1 - \frac{\frac{\alpha h}{2}}{t_k} \right) \nabla^2 f (X(t_k))\dot{X}(t_k)  + \mathcal O (h^2) = 0.
\end{equation*}
By neglecting the term of order $s=h^2$, and keeping the terms of order $h= \sqrt{s}$, we obtain the claimed inertial dynamic with Hessian driven damping.
This completes the proof.
\end{proof}


\section{The dynamical system perspective of the Ravine method}\label{sec:dyn-perp}

Let us examine successively the low then the high resolution ODE of the Ravine method.

\subsection{Low resolution ODE of \eqref{eq:RAG}}

According to the definition of the Ravine method, we have
\begin{eqnarray*}
y_{k+1}&=& y_k- s\nabla f (y_k) + \pa{1 - \frac{\alpha}{k+1}} \Big( y_k- s\nabla f (y_k)  - ( y_{k-1}- s\nabla f (y_{k-1}) )\Big)\\
&= & y_k  + \pa{1 - \frac{\alpha}{k+1}}  (y_k- y_{k-1}) - s\nabla f (y_k) -s \pa{1 - \frac{\alpha}{k+1}} \Big( \nabla f (y_k) -\nabla f (y_{k-1} )\Big).
\end{eqnarray*}
After division by $s=h^2$, we obtain,   equivalently
\begin{equation}\label{dyn-Rav-1}
\frac{ y_{k+1} - 2y_{k}+ y_{k-1} }{h^2}  + \frac{\alpha}{kh+h} \frac{y_k- y_{k-1}}{h} +\nabla f (y_k)
+ \pa{1 - \frac{\alpha}{k+1}} ( \nabla f (y_k) -\nabla f (y_{k-1} )) =0.
\end{equation}
We now follow an argument similar to the one in section~\ref{NtoAVD}. For each $k\in\N$, set $t_k=kh$,
and assume that  $y_k = Y(t_k)$ for some  smooth curve  $t \mapsto Y(t)$ defined for $t\geq t_0 >0$. Performing a Taylor expansion in powers of $h$, when $h$ is close to zero, of the different quantities involved in \eqref{dyn-Rav-1}, we obtain
\[
\ddot{Y}(t_k) + \frac{\alpha}{t_k} \dot{Y}(t_k) +  \nabla f (Y(t_k)) + \cO (h)  =0.
\]
Letting $h\to 0$ gives that $Y(\cdot)$ is a solution trajectory of \eqref{eq:AVD}.
We therefore obtain the same inertial dynamics as that associated with \eqref{eq:NAG}.

\subsection{High resolution ODE of \eqref{eq:RAG}}
By letting $h \to 0$ in \eqref{dyn-Rav-1}, the term $\bpa{1 - \frac{\alpha}{k+1}} ( \nabla f (y_k) -\nabla f (y_{k-1} ))$ disappears at the limit. Indeed, as we will see, this term is numerically important. To take it into account, we will perform a high resolution of \eqref{dyn-Rav-1}. The approach will be similar to that
developed in \cite{SDJS}, which will make appear the Hessian-driven damping in the associated continuous inertial equation. This is made precise in the following theorem.
\begin{theorem}\label{highresolution}
Assume that $f$ is $\cC^2$. The high resolution ODE with temporal step size $\sqrt{s}$ of \eqref{eq:RAG} 
%
gives the inertial dynamic with Hessian driven damping
\begin{equation}\label{ODE101}
\ddot{Y}(t) + \frac{\alpha}{t}\dot{Y}(t) + \sqrt{s}\nabla^2 f (Y(t)) \dot{Y}(t) + \left(1+ \frac{\alpha \sqrt{s} }{2t} \right) \nabla f (Y(t)) = 0.
\end{equation}
\end{theorem}

\begin{proof}
Let us start from the equivalent formulation \eqref{dyn-Rav-1} of \eqref{eq:RAG}, which we rewrite as follows
\begin{equation}\label{dyn-Rav-1c}
\frac{ y_{k+1} - 2y_{k}+ y_{k-1} }{h^2}  + \frac{\alpha}{k+1} \frac{y_k- y_{k-1}}{h^2} +\nabla f (y_k)
+ \frac{k+1-\alpha}{k+1} ( \nabla f (y_k) -\nabla f (y_{k-1} )) =0.
\end{equation}

\smallskip

\noindent Let us arrange the above formula, so as to prepare it for its analysis by Taylor expansion. After multiplying \eqref{dyn-Rav-1c} by $\frac{k+1}{k+1-\alpha}$, we get

\begin{small}
\begin{equation}\label{dyn-Rav-1d}
\frac{k+1}{k+1-\alpha}\frac{ y_{k+1} - 2y_{k}+ y_{k-1} }{h^2}  + \frac{\alpha}{k+1-\alpha} \frac{y_k- y_{k-1}}{h^2} +\frac{k+1}{k+1-\alpha} \nabla f (y_k)
+   \nabla f (y_k) -\nabla f (y_{k-1} ) =0.
\end{equation}
\end{small}

\smallskip

\noindent Notice then that 
$$
\frac{y_k- y_{k-1}}{h^2} = \frac{y_{k+1} -y_k}{h^2} - \frac{ y_{k+1} - 2y_{k}+ y_{k-1} }{h^2}. 
$$
Thus, \eqref{dyn-Rav-1d} can be  formulated equivalently as follows
\begin{small}
\begin{eqnarray*}
&&\Big(\frac{k+1}{k+1-\alpha} - \frac{\alpha}{k+1-\alpha} \Big) \frac{ y_{k+1} - 2y_{k}+ y_{k-1} }{h^2}  + \frac{\alpha}{k+1-\alpha} \frac{y_{k+1} -y_k}{h^2}\\
&& +\frac{k+1}{k+1-\alpha} \nabla f (y_k)
+   \nabla f (y_k) -\nabla f (y_{k-1} ) =0.
\end{eqnarray*}
\end{small}

\noindent After reduction we obtain, equivalently
\begin{small}
\begin{equation}\label{dyn-Rav-1e}
 \frac{ y_{k+1} - 2y_{k}+ y_{k-1} }{h^2}  + \frac{\alpha}{(k+1-\alpha)h} \frac{y_{k+1} -y_k}{h} +\left( 1+ \frac{\alpha}{k+1-\alpha}\right) \nabla f (y_k)
+   \nabla f (y_k) -\nabla f (y_{k-1} ) =0.
\end{equation}
\end{small}
\noindent Building on \eqref{dyn-Rav-1e}, we are now following  a device similar to the one developed in section \ref{NtoAVD}, and which uses Taylor expansions, but now taken at a higher order. 
For each $k\in\N$, set $t_k\eqdef(k+c)h $, where $c$ is a real parameter that will be adjusted later.
Assume that  $y_k = Y(t_k)$ for some  smooth curve  $t \mapsto Y(t)$ defined for $t\geq t_0 >0$. Performing a Taylor expansion in powers of $h$, when $h$ is close  to zero, of the different quantities involved in \eqref{dyn-Rav-1e}, we obtain
\begin{eqnarray}
y_{k+1} &=& Y(t_{k+1})= Y(t_k) + h \dot{Y}(t_k)+ \demi h^2 \ddot{Y}(t_k) + \frac{1}{6}h^3 \dddot{Y}(t_k)+ 
\cO (h^4) \label{taylor1c}\\
y_{k-1} &=& Y(t_{k-1})= Y(t_k) - h \dot{Y}(t_k)+ \demi h^2 \ddot{Y}(t_k)  - \frac{1}{6}h^3 \dddot{Y}(t_k)+ 
\cO (h^4) . \label{taylor2c}
\end{eqnarray}
By adding \eqref{taylor1c} and \eqref{taylor2c}  we obtain 
$$
\frac{ y_{k+1} - 2y_{k}+ y_{k-1} }{h^2} =   \ddot{Y}(t_k) + \cO (h^2).
$$
Moreover, \eqref{taylor1c} gives
$$
\frac{y_{k+1} - y_{k}}{h} =   \dot{Y}(t_k) + \demi h \ddot{Y}(t_k)  + \cO (h^2).
$$
\noindent  By Taylor expansion of $\nabla f$  we have
$$ \nabla f (y_k) -\nabla f (y_{k-1})=   
 \nabla^2 f (Y(t_k)) \dot{Y}(t_k)h + \cO \left(h^2\right).
$$
Plugging all of the above results into \eqref{dyn-Rav-1e}, we obtain
\begin{multline*}
\bpa{\ddot{Y}(t_k)+\cO (h^2)} +  \frac{\alpha}{(k+1-\alpha)h}\bpa{\dot{Y}(t_k)+ \demi h \ddot{Y}(t_k)  + \cO (h^2)} \\
+  \frac{k+1}{k+ 1-\alpha} \nabla f (Y(t_k)) + \bpa{h\nabla^2 f (Y(t_k)) \dot{Y}(t_k) + \cO \left( h^2\right)}=0.
\end{multline*}
After multiplication by $\frac{(k+1-\alpha)h}{\alpha} $, and reduction of the terms involving $\ddot{Y}(t_k)$,  we obtain 
\begin{eqnarray*}
\frac{h}{\alpha}\pa{k+1-\frac{\alpha}{2}}\ddot{Y}(t_k) +  \dot{Y}(t_k) +  \frac{(k+1)h}{\alpha} \nabla f (Y(t_k)) 
+ h\frac{(k+1-\alpha)h}{\alpha} \nabla^2 f (Y(t_k)) \dot{Y}(t_k) +\cO (h^3)  =0.
\end{eqnarray*}
Dividing by $\frac{h}{\alpha}(k+1-\frac{\alpha}{2})$ yields 
\begin{multline*}
\ddot{Y}(t_k) + \frac{\alpha}{(k+1-\frac{\alpha}{2})h} \dot{Y}(t_k) + \left(1+   \frac{\frac{\alpha}{2}}{k+1-\frac{\alpha}{2}} \right)\nabla f (Y(t_k)) \\
+ h\left(1-\frac{\frac{\alpha}{2}}{k+1-\frac{\alpha}{2}} \right)\nabla^2 f (Y(t_k)) \dot{Y}(t_k) +\cO (h^2)  =0.
\end{multline*}
Take $c=1-\frac{\alpha}{2}$ and thus $t_k \eqdef (k+1-\frac{\alpha}{2})h$. We obtain 
\begin{eqnarray*}
&&\ddot{Y}(t_k) + \frac{\alpha}{t_k} \dot{Y}(t_k) +  \left(   1+ 
\frac{\alpha h }{2t_k} \right) \nabla f (Y(t_k)) 
+ h\left(1-   \frac{\alpha h}{2t_k} \right) \nabla^2 f (Y(t_k)) \dot{Y}(t_k) +\cO (h^2)  =0.
\end{eqnarray*}
By neglecting the term of order $s=h^2$, and keeping the terms of order $h= \sqrt{s}$, we obtain the claimed inertial dynamic with Hessian driven damping.
This completes the proof. 
\end{proof}

A few remarks are in order.

\begin{remark}
The high resolution ODE's of \eqref{eq:RAG} and \eqref{eq:NAG} have the same structure but have also differences. First, they are given in terms of two different variables: $x$ for \eqref{eq:NAG} and $y$ for \eqref{eq:RAG}. Observe also that to get the high resolution ODE \eqref{highresODENAG}, the Hessian appears after applying an extra Taylor expansion on the gradient, which is reminiscent of our discussion on the implicit Hessian damping ODE \eqref{eq:odetwo}. On the other hand, the Hessian appears from an explicit discretization in the ODE \eqref{ODE101} associated to \eqref{eq:RAG}.  
\end{remark}

\begin{remark}
Recall that the dynamic with Hessian driven damping \eqref{eq:DINAVD} which supports the inertial gradient algorithm \eqref{eq:IGAHD} developed in \cite{ACFR} is given by
\begin{equation}\label{ODE10}
\ddot{x}(t) + \frac{\alpha}{t}\dot{x}(t) + \beta \nabla^2 f (x(t)) \dot{x}(t) + \left(1 + \frac{\beta}{t} \right) \nabla f (x(t))  =0.
\end{equation}
It is in accordance with the high resolution ODE's \eqref{highresODENAG} and \eqref{ODE101} of respectively \eqref{eq:NAG} and \eqref{eq:RAG}, and allows to interpret the Hessian driven damping coefficient $\beta$ of \eqref{ODE10} as a temporal step size. This also paves the way to proving fast convergence to zero of the gradients which is the subject of the next section.
\end{remark}

\section{Fast convergence to zero of the gradients for \eqref{eq:RAG} and \eqref{eq:NAG}}

Let us come to another central point of our study, which concerns the fast convergence towards zero of the gradients. Closely related results were obtained in \cite{ACFR}, and \cite{SDJS,Shi19} in different contexts and discretizations. In \cite{ACFR}, the algorithm considered is \eqref{eq:IGAHD}, whose underlying dynamic is \eqref{eq:DINAVD}, but the discretization is different. It is inspired by the Nesterov scheme, which  contrasts with the  Ravine method which is based on an explicit discretization scheme. In \cite{SDJS}, the structure of the algorithm is the same as that of \eqref{eq:RAG}, but the extrapolation parameter is different, which requires an independent proof. Other discretizations are also discussed in \cite{Shi19} after a first-order equivalent reformulation of \eqref{ODE10}. The fast convergence rates on the gradients shown in \cite{SDJS,Shi19} turn out to be weaker than ours. Observe also that developing the Ravine method with a general extrapolation parameter, as was done by Attouch and Cabot in \cite{AC2} for the accelerated gradient method, is an interesting research venue that we leave to a future work.

\subsection{The case of \eqref{eq:RAG}}
For the convenience of the reader, we give a self-contained proof, which is based on Lyapunov analysis. We will rely on the following equivalent formulation of the Ravine method which was obtained in \eqref{dyn-Rav-1}, and which gives rise to the dynamic interpretation with the damping driven by the Hessian:
\begin{equation}\label{eq:RAGequiv}
\frac{ y_{k+1} - 2y_{k}+ y_{k-1} }{h^2}  + \frac{\alpha}{(k+1-\alpha)h} \frac{y_{k+1} -y_k}{h} +\left( 1+ \frac{\alpha}{k+1-\alpha}\right) \nabla f (y_k) + \nabla f (y_k) -\nabla f (y_{k-1} ) =0.
\end{equation}
To make the notations shorter, it is convenient to introduce the 
discrete velocity $v_k$ which is defined for each  $k\in \N$ by
\[
v_k=\frac{1}{h}(y_{k+1} - y_{k}).
\]
So, the constitutive equation \eqref{eq:RAGequiv} can be equivalently written as
\begin{equation}\label{dyn-Rav-110}
 v_k - v_{k-1} =   - \frac{\alpha}{(k+1-\alpha)} v_k  - h  \Big( \nabla f (y_k) -\nabla f (y_{k-1} \Big) - h\frac{k+1}{k+1-\alpha} \nabla f (y_k).
\end{equation}

\noindent 
Given  $x^\star \in \argmin_{\cH}(f)$,  our Lyapunov analysis is based on the energy sequence $\seq{E_k}$  defined by
\begin{eqnarray}
E_k &\eqdef& h^2(k+2-\alpha)(k+1)( f (y_k)-  f( x^\star) ) +\frac{1}{2}\|z_k\|^2 \label{Lyap-function-grad1b}\\
z_k &\eqdef& (\alpha-1)(y_{k+1}  - x^\star) + h(k+2-\alpha)\Big( v_{k} + h \nabla f( y_{k})\Big).\label{Lyap-function-grad2b}
\end{eqnarray}

\begin{theorem}\label{pr.decay_E_k}
Let  $f: \cH \to \R$ be a $\cC^1$ convex function whose gradient is $L$-Lipschitz continuous. Let $\seq{y_k}$ be the sequence generated by \eqref{eq:RAG}, where $\alpha \geq 3$ and $sL < 1$. Then the  sequence  $\seq{E_k}$  defined by {\rm\eqref{Lyap-function-grad1b}-\eqref{Lyap-function-grad2b}} is non-increasing, and the following convergence rate is satisfied:
\[
\sum_k  k^2 \| \nabla f (y_k) \|^2 < +\infty .
\]
In addition, when $\alpha > 3$,
\[
\sum_k  k( f(y_k)-  f( x^\star) )  < +\infty .
\]
\end{theorem}
\begin{proof}
By definition of $E_k$, we have
\begin{eqnarray}
E_{k+1} - E_k &=& h^2(k+2-\alpha)(k+1)( f (y_{k+1})-  f (y_{k}) ) +h^2 (2k+4-\alpha)( f (y_{k+1})-  f( x^\star) ) \nonumber \\
&+&\frac{1}{2}\|z_{k+1}\|^2 -\frac{1}{2}\|z_{k}\|^2   \label{Lyap-10} . 
\end{eqnarray}
Let us compute this last expression with the help of the elementary identity
\begin{equation}\label{elem-ineq}
\frac{1}{2}\|z_{k+1}\|^2 -\frac{1}{2}\|z_k\|^2  = \left\langle  z_{k+1} -z_{k} , z_{k+1} \right\rangle - \frac{1}{2}\|z_{k+1} - z_{k}\|^2  .
\end{equation}
First observe that the constitutive equation \eqref{dyn-Rav-110} gives
\begin{equation}\label{transform}
( v_k +h \nabla f (y_k))  -( v_{k-1}+h \nabla f (y_{k-1})) =   - \frac{\alpha}{(k+1-\alpha)} v_k 
 -h\ \frac{k+1}{k+1-\alpha} \nabla f (y_k).
\end{equation}
Therefore,
\begin{equation*}
(k+1-\alpha)( v_k +h \nabla f (y_k))  -(k+1-\alpha)( v_{k-1}+h \nabla f (y_{k-1}) =   - \alpha v_k 
 -h (k+1) \nabla f (y_k).
\end{equation*}
Equivalently,
\begin{equation}\label{dyn-Rav-112}
(k+1)( v_k +h \nabla f (y_k))  -(k+1-\alpha)( v_{k-1}+h \nabla f (y_{k-1})) =    
 -h (k+1 -\alpha) \nabla f (y_k).
\end{equation}
Using successively the definition of $z_k$ and \eqref{dyn-Rav-112}, we obtain
\begin{eqnarray*}
z_{k+1} - z_{k}&=& (\alpha-1)(y_{k+2}  - y_{k+1})\\
&& + h(k+3-\alpha)\Big( v_{k+1} + h \nabla f( y_{k+1})  \Big) - h(k+2-\alpha)\Big( v_{k} + h \nabla f( y_{k})  \Big)\\
&=& h(\alpha-1) v_{k+1} + h(k+3-\alpha)\Big( v_{k+1} + h \nabla f( y_{k+1})  \Big) - h(k+2-\alpha)\Big( v_{k} + h \nabla f( y_{k})  \Big)  \\
&=& h(k+2)\Big( v_{k+1} + h \nabla f( y_{k+1})  \Big) - h(k+2-\alpha)\Big( v_{k} + h \nabla f( y_{k})  \Big)  -h^2(\alpha-1)\nabla f( y_{k+1})\\
&=& - h^2 (k+2 -\alpha) \nabla f (y_{k+1})-h^2(\alpha-1)\nabla f( y_{k+1})\\
&=& - h^2 (k+1) \nabla f (y_{k+1}).
\end{eqnarray*}
Plugging this into \eqref{elem-ineq}, we obtain
\begin{eqnarray*}
&&\frac{1}{2}\|z_{k+1}\|^2 -\frac{1}{2}\|z_k\|^2  =  - \frac{1}{2}h^4 (k+1)^2 \| \nabla f (y_{k+1}) \|^2 \\
&& - h^2( k+1) \left\langle   \nabla f (y_{k+1}) , (\alpha-1)(y_{k+1}  - x^\star) + h(k+2-\alpha)\Big( v_{k} + h \nabla f( y_{k}) \Big) - h^2 (k+1) \nabla f (y_{k+1}) \right\rangle\\
&&= \demi  h^4 (k+1)^2 \| \nabla f (y_{k+1}) \|^2 \\
&& -  h^2 (k+1) \left\langle   \nabla f (y_{k+1}) , (\alpha-1)(y_{k+1}  - x^\star) + h(k+2-\alpha)\Big( v_{k} + h \nabla f( y_{k}) \Big)\right\rangle.
\end{eqnarray*}
Let us rearrange this last expression so that we have terms involving only $\nabla f (y_{k+1})$. For this, we use \eqref{transform} that we write as follows
\[
v_k +h \nabla f (y_k)) = v_{k+1} +h \nabla f (y_{k+1}) + \frac{\alpha}{k+2-\alpha} v_{k+1} +h \frac{k+2}{k+2-\alpha}\nabla f (y_{k+1}).
\]
Therefore,
\begin{eqnarray*}
&&(\alpha-1)(y_{k+1}  - x^\star) + h(k+2-\alpha)\Big( v_{k} + h \nabla f( y_{k})\Big)\\
&&= (\alpha-1)(y_{k+1}  - x^\star) + h(k+2-\alpha)\Big( v_{k+1} +h \nabla f (y_{k+1}) + \frac{\alpha}{k+2-\alpha}  v_{k+1} 
 +h \frac{k+2}{k+2-\alpha} \nabla f (y_{k+1}) \Big)\\
 &&= (\alpha-1)(y_{k+1}  - x^\star) + h(k+2) v_{k+1} 
 +h^2 (2k+4-\alpha) \nabla f (y_{k+1}).
\end{eqnarray*}
Collecting the above results we obtain
\begin{multline*}
\frac{1}{2}\|z_{k+1}\|^2 -\frac{1}{2}\|z_k\|^2  =   \frac{1}{2}h^4 (k+1)^2 \| \nabla f (y_{k+1}) \|^2 \\
 - h^2 (k +1)\left\langle   \nabla f (y_{k+1}) ,  (\alpha-1)(y_{k+1}  - x^\star) + h(k+2) v_{k+1} 
 +h^2 (2k+4-\alpha) \nabla f (y_{k+1})
\right\rangle .
\end{multline*}
Combining this inequality with \eqref{Lyap-10} we get
\begin{equation}\label{basicEk}
\begin{aligned}
E_{k+1} - E_k &= h^2(k+2-\alpha)(k+1)( f (y_{k+1})-  f (y_{k}) ) +h^2 (2k+4-\alpha)( f (y_{k+1})-  f( x^\star) ) \\
&+\frac{1}{2}h^4 (k+1)^2 \| \nabla f (y_{k+1}) \|^2  \\
&- h^2 (k+1) \left\langle   \nabla f (y_{k+1}) ,  (\alpha-1)(y_{k+1}  - x^\star) + h(k+2) v_{k+1}  +h^2 (2k+4-\alpha) \nabla f (y_{k+1})
\right\rangle. 
\end{aligned}
\end{equation}
According to the basic gradient inequality for convex differentiable functions whose gradient is $L$-Lipschitz continuous, we have
\begin{eqnarray*}
  f (y_{k}) &\geq&  f (y_{k+1})+  \left\langle   \nabla f (y_{k+1}) ,  y_k - y_{k+1} \right\rangle + \frac{1}{2L}\|\nabla f (y_{k+1}) -\nabla f (y_{k})  \|^2 ,\\
f( x^\star) &\geq&  f (y_{k+1})+  \left\langle   \nabla f (y_{k+1}) ,  x^\star - y_{k+1} \right\rangle.
\end{eqnarray*}
Combining the above inequalities with \eqref{basicEk}, we obtain
\begin{eqnarray*}
 && E_{k+1} - E_k \leq -h^2(k+2-\alpha)(k+1)\Big(\left\langle   \nabla f (y_{k+1}) ,  y_k - y_{k+1} \right\rangle + \frac{1}{2L}\|\nabla f (y_{k+1}) -\nabla f (y_{k})  \|^2 \Big) \\
&& +h^2 (2k+4-\alpha)( f (y_{k+1})-  f( x^\star) ) + h^2 (k+1) (\alpha -1)( f( x^\star) - f (y_{k+1}) )  \\
&& +\frac{1}{2}h^4 (k+1)^2 \| \nabla f (y_{k+1}) \|^2 - h^2 (k+1) \left\langle   \nabla f (y_{k+1}) ,   h(k+2) v_{k+1}  +h^2 (2k+4-\alpha) \nabla f (y_{k+1})
\right\rangle. 
\end{eqnarray*}
Equivalently,
\begin{eqnarray*}
&& E_{k+1} - E_k \leq h^2(k+2-\alpha)(k+1)\Big(\left\langle   \nabla f (y_{k+1}) ,  h v_k  \right\rangle - \frac{1}{2L}\|\nabla f (y_{k+1}) -\nabla f (y_{k})  \|^2 \Big) \\
&& -h^2 \Big( k(\alpha-3) +2\alpha -5 \Big) ( f (y_{k+1})-  f( x^\star) ) - h^2 (k +1)\left\langle   \nabla f (y_{k+1}) ,   h(k+2) v_{k+1}  \right\rangle \\
&& -\frac{1}{2}h^4 (k+1) (3k +7 -2\alpha) \| \nabla f (y_{k+1}) \|^2 . 
\end{eqnarray*}
Let us put together the terms involving the scalar product with 
$\nabla f (y_{k+1})$. We get
\begin{eqnarray*}
&& E_{k+1} - E_k \leq -h^2(k+2-\alpha)(k+1) \frac{1}{2L}\|\nabla f (y_{k+1}) -\nabla f (y_{k})  \|^2  \\
&&+ h^3(k+1)\left\langle   \nabla f (y_{k+1}) , (k+2-\alpha) v_k -(k+2) v_{k+1}  \right\rangle \\
&& -h^2 \Big( k(\alpha-3) +2\alpha -5 \Big) ( f (y_{k+1})-  f( x^\star) )  
-\frac{1}{2}h^4 (k+1) (3k +7 -2\alpha) \| \nabla f (y_{k+1}) \|^2 . 
\end{eqnarray*}
According to \eqref{dyn-Rav-112} we have 
\begin{equation}\label{dyn-Rav-112-bb}
(k+2-\alpha)( v_{k}+h \nabla f (y_{k}))-(k+2)( v_{k+1} +h \nabla f (y_{k+1}))   =    
 h (k+2 -\alpha) \nabla f (y_{k+1}).
 \end{equation}
Therefore
\begin{equation}\label{dyn-Rav-112-cc}
(k+2-\alpha) v_{k}-(k+2) v_{k+1} =  - h (k+2 -\alpha) \nabla f (y_{k}) +h(2k +4-\alpha) \nabla f (y_{k+1}).
\end{equation}
Combining the above results we get
\begin{eqnarray*}
&& E_{k+1} - E_k  +h^2 \Big( k(\alpha-3) +2\alpha -5 \Big) ( f (y_{k+1})-  f( x^\star) )\\
&&\leq -\frac{h^2}{2L}(k+2-\alpha)(k+1) \|\nabla f (y_{k+1}) -\nabla f (y_{k})  \|^2  \\
&&+ h^3(k+1)\Big(\left\langle   \nabla f (y_{k+1}) ,
- h (k+2 -\alpha) \nabla f (y_{k}) +h(2k +4-\alpha) \nabla f (y_{k+1}) \right\rangle \Big)\\
&& -\frac{1}{2}h^4 (k+1) (3k +7 -2\alpha) \| \nabla f (y_{k+1}) \|^2 . 
\end{eqnarray*}
Equivalently
\begin{equation}\label{eq:basic_Lyap_2611}
E_{k+1} - E_k  + h^2 \Big( k(\alpha-3) +2\alpha -5 \Big) ( f (y_{k+1})-  f( x^\star) ) + R( \nabla f (y_{k}), \nabla f (y_{k+1})) \leq 0 ,
\end{equation}
where
\begin{eqnarray*}
R(X,Y)&=& \frac{h^2}{2L}(k+2-\alpha)(k+1) \|Y-X  \|^2  +\frac{1}{2}h^4 (k+1) (3k +7 -2\alpha) \| Y \|^2 \\
&-& h^3(k+1)\Big(\left\langle   Y ,- h (k+2 -\alpha) X +h(2k +4-\alpha) Y \right\rangle \Big).
\end{eqnarray*}
To conclude, we just need to prove that the  quadratic form $R$ is positive definite. A simple procedure consists in computing
$\inf_X R(X,Y)$. For fixed $Y$, the minimum of $R(\cdot,Y)$ is achieved at $\bar{X}$ with $\bar{X}-Y= -h^2 LY$. Therefore
\begin{eqnarray*}
\inf_X R(X,Y) &=& \frac{h^4 L}{2}(k+2-\alpha)(k+1) h^2\|Y\|^2  +\frac{1}{2}h^4 (k+1) (3k +7 -2\alpha) \| Y \|^2\\
 &-& h^3(k+1)\Big(\left\langle   Y ,
  - h (k+2 -\alpha) (1-h^2 L)Y +h(2k +4-\alpha) Y \right\rangle.  
\end{eqnarray*}
After reduction, we get
\begin{eqnarray*}
\inf_X R(X,Y)
&=& \frac{h^4 (k+1)}{2} \|Y\|^2 \Big( k(1-Lh^2) + Lh^2 (\alpha-2) + 3-2\alpha \Big).
\end{eqnarray*}
Returning to \eqref{eq:basic_Lyap_2611} we obtain
\begin{eqnarray*}
&& E_{k+1} - E_k  +h^2 \Big( k(\alpha-3) +2\alpha -5 \Big) ( f (y_{k+1})-  f( x^\star) ) \\
&& + \frac{h^4 (k+1)}{2} \|  \nabla f (y_{k+1})\|^2 \Big( k(1-Lh^2) + Lh^2 (\alpha-2) + 3-2\alpha \Big)\leq 0.
\end{eqnarray*}
So, when $\alpha \geq 3$ and $Lh^2 \in ]0,1[$, we obtain that $\seq{E_k}$ is a non-negative non-increasing sequence, hence convergent. By summing the above inequalities over $k$, we finally obtain
\[
\sum_k  k^2  \|  \nabla f (y_{k+1})\|^2  < +\infty.
\]
Also note that when $\alpha >3$ we obtain 
\[
\sum_k  k( f (y_{k+1})-  f( x^\star) )  < +\infty.
\]
\end{proof}

\begin{remark}
It is easy to see that our result entails
\[
\min_{1 \leq i \leq k} \|\nabla f(y_i)\|^2 = \cO\pa{\frac{1}{k^3}} ,
\]
which recovers the rate found in \cite{SDJS,Shi19}.
\end{remark}

\subsection{The case of \eqref{eq:NAG}}
We have the following fast convergence of gradients to zero for \eqref{eq:NAG}. Surprisingly, this result has never been established before in the literature.
\begin{theorem}\label{thm:fastgradNAG}
Let $f: \cH \to \R$ be a $\cC^1$ convex function whose gradient is $L$-Lipschitz continuous. Let $\seq{x_k}$ be the sequence generated by \eqref{eq:NAG}, where $\alpha \geq 3$ and $sL < 1$. Then the following convergence rate is satisfied:
\[
\sum_k  k^2 \| \nabla f (x_k) \|^2 < +\infty  .
\]
In addition, when $\alpha > 3$,
\[
\sum_k  k( f (x_{k})-  f( x^\star) )  < +\infty .
\]
\end{theorem}

\begin{proof}
The analysis relies on the energy function
\begin{equation*}
\cE_k(t)=t_k^2(f(x_k)-\min_{\cH} f)+\frac{1}{2s}\|x_{k-1} -x^\star + t_k(x_k-x_{k-1})\|^2,
\end{equation*}
where $t_k \eqdef \frac{k-1}{\alpha-1}$. We will also need the following refined version of the descent lemma
\begin{equation}\label{eq:impdesclem}
f(y - s \nabla f (y)) \leq f (x) + \left\langle  \nabla f (y), y-x \right\rangle -\frac{s}{2} \|  \nabla f (y) \|^2 -\frac{s}{2} \| \nabla f (x)- \nabla f (y) \|^2 
\end{equation}
valid for any $(x,y) \in \cH^2$ and $s \in ]0,1/L]$. 

Let us write successively \eqref{eq:impdesclem} at $y=y_k$ and $x=x_k$, then at $y=y_k$ and $x=x^\star$. Recalling that $x_{k+1}=y_k-s\nabla f(y_k)$, we obtain the two inequalities
\begin{eqnarray}
f(x_{k+1})&\leq& f(x_k)+\langle \nabla f(y_k),y_k-x_k\rangle -\frac{s}{2}\|\nabla f(y_k)\|^2 -\frac{s}{2} \| \nabla f (x_k)- \nabla f (y_k) \|^2 , \label{eq.1}\\
f(x_{k+1})&\leq& \min_{\cH} f+\langle\nabla f (y_k),y_k-x^\star\rangle -s\|\nabla f(y_k)\|^2 . \label{eq.2}
\end{eqnarray}
Multiplying \eqref{eq.1} by $t_{k+1}-1\geq 0$, then adding \eqref{eq.2}, we derive that
\begin{eqnarray}\label{eq.combinaison_1}
t_{k+1} f(x_{k+1})&\leq& (t_{k+1}-1)f(x_k)+\min_{\cH} f 
+\langle \nabla f(y_k),(t_{k+1}-1)(y_k-x_k)+y_k-x^\star\rangle \nonumber\\
&-&\frac{s}{2}(t_{k+1}+1)\|\nabla f(y_k)\|^2 -\frac{s}{2}(t_{k+1}-1)\| \nabla f (x_k)- \nabla f (y_k) \|^2 . 
\end{eqnarray}
Observing that $t_k$ verifies the identity $\alpha_k \eqdef 1- \dfrac{\alpha}{k} = \dfrac{t_k -1}{t_{k+1}}$, we obtain
\begin{eqnarray*}
(t_{k+1}-1)(y_k-x_k)+y_k&=&t_{k+1}\,y_k-(t_{k+1}-1)x_k\\
&=&x_k+t_{k+1}\,\alpha_k(x_k-x_{k-1})\\
&=&x_{k-1}+(1+t_{k+1}\,\alpha_k)(x_k-x_{k-1})\\
&=&x_{k-1}+t_k(x_k-x_{k-1}).
\end{eqnarray*}
Setting $z_k=x_{k-1}+t_k(x_k-x_{k-1})$, we then deduce from \eqref{eq.combinaison_1} that
\begin{eqnarray}\label{eq.majo_t_k+1_Theta(x_k+1)}
t_{k+1}(f(x_{k+1})-\min_{\cH} f)&\leq& (t_{k+1}-1)(f(x_k)-\min_{\cH} f)+\langle \nabla f (y_k),z_k-x^\star\rangle   \nonumber\\
&& -\frac{s}{2}(t_{k+1}+1)\|\nabla f(y_k)\|^2 -\frac{s}{2}(t_{k+1}-1)\| \nabla f (x_k)- \nabla f (y_k) \|^2 .
\end{eqnarray}
On the other hand, observe that
\begin{eqnarray*}
z_{k+1}-z_k&=&x_k+t_{k+1}(x_{k+1}-x_k)-x_{k-1}-t_k(x_k-x_{k-1})\\
&=&t_{k+1}(x_{k+1}-x_k)-(t_k-1) (x_k-x_{k-1})\\
&=&t_{k+1}(x_{k+1}-x_k-\alpha_k(x_k-x_{k-1})) \\
&=&t_{k+1}(x_{k+1}-y_k)=-s\, t_{k+1} \nabla f(y_k).
\end{eqnarray*}
It ensues that
\begin{equation}\label{eq.variations_norme(z_x-x*)_rescale}
\|z_{k+1}-x^\star\|^2=\|z_{k}-x^\star\|^2-2s t_{k+1}\langle \nabla f(y_k),z_k-x^\star\rangle +s^2t_{k+1}^2\| \nabla f(y_k)\|^2.
\end{equation}
By using this equality in \eqref{eq.majo_t_k+1_Theta(x_k+1)}, we find
\begin{eqnarray}\label{eq.majo_t_k+1_Theta(x_k+1)_b}
&&t_{k+1}(f(x_{k+1})-\min_{\cH} f)\leq (t_{k+1}-1)(f(x_k)-\min_{\cH} f)+\frac{1}{2st_{k+1}}(\|z_{k}-x^\star\|^2-\|z_{k+1}-x^\star\|^2)  \nonumber\\
&&\qquad  +\frac{st_{k+1}}{2} \| \nabla f(y_k)\|^2 -\frac{s}{2}(t_{k+1}+1)\|\nabla f(y_k)\|^2 -\frac{s}{2}(t_{k+1}-1)\| \nabla f (x_k)- \nabla f (y_k) \|^2 .
\end{eqnarray}
which is equivalent to
\begin{multline*}
t_{k+1}^2(f(x_{k+1})-\min_{\cH} f)+\frac{1}{2s}\|z_{k+1}-x^\star\|^2
\leq (t_{k+1}^2-t_{k+1})(f(x_k)-\min_{\cH} f)+\frac{1}{2s}\|z_{k}-x^\star\|^2\\
-\frac{st_{k+1}^2}{2} \Big(  \|\nabla f(y_k)\|^2 +(t_{k+1}-1)\| \nabla f (x_k)- \nabla f (y_k) \|^2  \Big).
\end{multline*}
Using the expression of the sequence $\seq{\cE_k}$, we obtain
\begin{multline*}   
\cE_{k+1}\leq\cE_k+(t_{k+1}^2-t_k^2-t_{k+1})(f(x_k)-\min_{\cH} f)
-\frac{st_{k+1}^2}{2} \Big(  \|\nabla f(y_k)\|^2 +(t_{k+1}-1)\| \nabla f (x_k)- \nabla f (y_k) \|^2  \Big).
\end{multline*}
Elementary algebraic computation gives
\[
\|\nabla f(y_k)\|^2 +(t_{k+1}-1)\| \nabla f (x_k)- \nabla f (y_k) \|^2  \geq \frac{t_{k+1}-1}{t_{k+1}} \|\nabla f(x_k)\|^2 .
\]
We finally obtain
\begin{equation}\label{eq:cEkest}   
\cE_{k+1} +  \frac{s}{2} t_{k+1}(t_{k+1}-1)\|\nabla f(x_k)\|^2 \leq \cE_k+(t_{k+1}^2-t_k^2-t_{k+1})(f(x_k)-\min_{\cH} f) .
\end{equation}
By definition of $t_k$, we have $t_{k+1}^2-t_k^2-t_{k+1} = -\frac{k(\alpha-3)+1}{(\alpha-1)^2} \leq 0$ for $\alpha \geq 3$. Thus, $\seq{\cE_k}$ is a non-increasing non-negative function, and summing \eqref{eq:cEkest}, we obtain
\[
\sum_k  k(k+1-\alpha)\|\nabla f(x_k)\|^2 <+\infty . 
\]
When $\alpha > 3$, summing again we get the second claimed estimate
\[
\sum_k  (k(\alpha-3)+1)(f(x_k)-\min_{\cH} f) <+\infty .
\]
\end{proof}

\begin{remark}
Another way to show Theorem~\ref{thm:fastgradNAG} is to use Theorem~\ref{pr.decay_E_k} and the equivalence result in Theorem~\ref{Nest_Ravine} between \eqref{eq:NAG} and \eqref{eq:RAG}.
\end{remark}

\begin{remark}
Theorem~\ref{thm:fastgradNAG} yields the rate
\[
\min_{1 \leq i \leq k} \|\nabla f(x_i)\|^2 = \cO\pa{\frac{1}{k^3}} ,
\]
which matches the complexity bound in \cite[Item~2.]{Nesterov12}. In \cite[Item~3.]{Nesterov12}, a better complexity bound is obtained by applying \eqref{eq:NAG} with $\alpha=3$ to a Tikhonov regularization of $f$ with an asymptotically vanishing parameter. From those complexity bounds, one can straightforwardly show that this parameter has to scale as $\cO\pa{\pa{\frac{\log k}{k}}^2}$ leading to a rate on the gradients
\[
\|\nabla f(x_k)\|^2 = \cO\pa{\pa{\frac{\log k}{k}}^4} .
\]
This is in agreement with our result in Theorem~\ref{thm:fastgradNAG}. On the other hand, one infers from our result that $\|\nabla f(x_k)\|^2$ must decrease at least as fast as $\cO\pa{\frac{1}{k^3(\log k)^\nu}}$, for $\nu > 1$.
\end{remark}

\section{The Ravine accelerated proximal gradient method}

Let us now extend the Ravine method \eqref{eq:RAG} to the case of additively structured "smooth + non-smooth" convex minimization problems
\begin{equation}\label{basic-min}
\min_{x \in \cH} \left\lbrace  \theta(x) \eqdef f(x) + g(x) \right\rbrace,
\end{equation}
where we make the following assumptions
\begin{equation}\tag{H}\label{eq:assum}
\begin{cases}
f: \cH \rightarrow \R  \text{ is a $\cC^1$ convex function and } \nabla f   \text{ is } \text{$L$-Lipschitz continuous; }     \\
g: \cH \rightarrow \mathbb R \cup \left\lbrace +\infty\right\rbrace  \text{ is proper, convex and lower semicontinuous};  \\
S \eqdef \argmin_{\cH} (\theta) \neq \emptyset. 
\end{cases}
\end{equation}

Note that by the above assumptions, $\theta:  \cH \rightarrow \mathbb R \cup \left\lbrace +\infty\right\rbrace  $ is a proper, convex and lower semicontinuous function. A natural extension of \eqref{eq:NAG} to this setting is
%
\begin{equation}
\label{eq:IPG}
\begin{cases}
y_k&= x_{k} + \left(1- \frac{\alpha}{k} \right)( x_{k}  - x_{k-1}) \\
x_{k+1}&= \prox_{sg}\pa{y_k- s\nabla f (y_k)}
\end{cases}
\end{equation}
which generalizes FISTA \cite{BT}.
Let us recall the convergence properties of this algorithm, which are valid under the assumption $sL \leq 1$:
\begin{itemize}
\item $\alpha = 3$: According to \cite{BT}, one has
\[
\theta(x_k)-  \min_{\cH} \theta = \cO \pa{\frac{1}{k^2}}  \mbox{ as } k \to +\infty.
\]
\item $\alpha> 3$: According to \cite{CD} and \cite{AP}
\[
\theta(x_k)-  \min_{\cH} \theta = o\pa{\frac{1}{k^2}}, \quad \|x_k - x_{k-1}\| =o\left(\frac{1}{k}\right) \mbox{ as } k \to +\infty, \qandq \wlim x_k \in S .
\]
\item $\alpha \leq 3$: According to \cite{AAD} and \cite{ACR-subcrit}
\[
\theta(x_k)-  \min_{\cH} \theta = \cO \left(k^{-\frac{2\alpha}{3}}\right) .
\]
\end{itemize}   

Before presenting our algorithm, let us introduce the operator $T_s: \cH \to \cH$ defined by
\begin{equation}\label{def:T-operator}
T_s(y) = \frac{1}{s} \Big(y - \mbox{prox}_{s g} \left( y- s\nabla f (y)\right)\Big).
\end{equation}
Note that solving \eqref{basic-min} is equivalent to find a zero of $T_s$. In addition,  the operator $T_s$ reduces to the gradient operator $\nabla f$ when $g=0$.
Thus the algorithm \eqref{eq:IPG} can be formulated in an equivalent way in the following form 
\begin{equation*}
\begin{cases}
y_k&= x_{k} + \left(1- \frac{\alpha}{k} \right)( x_{k}  - x_{k-1}) \\
x_{k+1}&= y_k -s T_s(y_k).
\end{cases}
\end{equation*}
As a consequence, all the algebraic developments concerning the Ravine method \eqref{eq:RAG} can be extended to the structured additive setting, by just replacing the gradient operator $\nabla f$ by $T_s$. Indeed, one can easily show that for $sL \leq 1$, the two operators share the following properties: monotonicity, co-coercivity and Lipschitz continuity which play a central role in the Lyapunov analysis.

With the help of this analogy, we are now in position to introduce the Ravine Accelerated Proximal Gradient algorithm (\eqref{eq:RAPG} for short):
\begin{myframedeq}[0.6\linewidth]
\begin{equation}\tag{\RAPG{\alpha}}\label{eq:RAPG}
\begin{cases}
w_{k}   &= y_k -s T_s(y_k)\\
y_{k+1} &= w_{k} + \left(1- \frac{\alpha}{k} \right)( w_{k}  - w_{k-1}) .
\end{cases}
\end{equation}
\end{myframedeq}

Figure~\ref{figRAP} gives some geometric insight into the scheme \eqref{eq:RAPG}. 
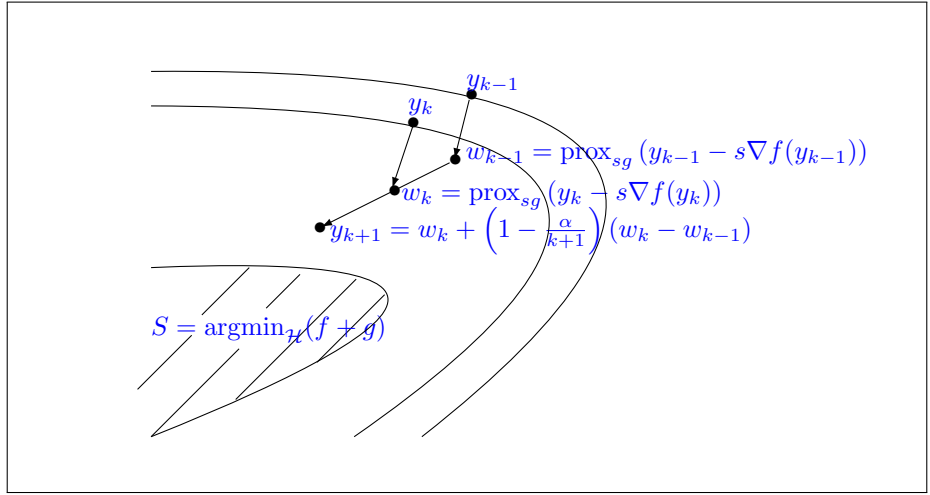
\begin{figure}[hbt!]
 \centering
 \fbox{\begin{minipage}{12cm}

\setlength{\unitlength}{9cm}
\begin{picture}(1,0.7)(-0.3,0.08)

\put(0.364,0.647){$\bullet$}
\put(0.37,0.648){\vector(-1,-4){0.022}}

\put(0.34,0.55){$\bullet$}

\put(0.278,0.605){$\bullet$}
\put(0.287,0.61){\vector(-1,-3){0.031}}

\put(0.25,0.505){$\bullet$}

\put(0.341,0.555){\vector(-2,-1){0.19}}

\put(0.025,0.205){\line(1,1){0.178}}
\put(0.145,0.26){\line(1,1){0.1}}
\put(-0.1,0.15){\line(1,1){0.13}}
\put(0.07,0.34){\line(1,1){0.05}}

\put(-0.12,0.22){\line(1,1){0.08}}
\put(-0.025,0.33){\line(1,1){0.07}}

\put(0.14,0.45){$\bullet$}

\tcb{
\put(0.366,0.669){$y_{k-1}$}
\put(0.28,0.633){$y_{k}$}
\put(0.365,0.56){$w_{k-1} =   \mbox{prox}_{s g} \left( y_{k-1}- s\nabla f (y_{k-1})\right) $}
\put(0.27,0.5){$w_{k} =  \mbox{prox}_{s g} \left( y_{k}- s\nabla f (y_{k})\right) $}
\put(0.164,0.445){$y_{k+1}=  w_k + \pa{1 - \frac{\alpha}{k+1}} \left( w_k - w_{k-1}\right)$}
\put(-0.1,0.3){$S = \argmin_{\cH}(f+g)$}
}

\qbezier(-0.1,0.15)(0.6,0.43)(-0.1,0.4)
\qbezier(0.3,0.15)(1.0,0.7)(-0.1,0.69)
\qbezier(0.2,0.15)(0.9,0.637)(-0.1,0.639)

\end{picture}
\end{minipage}}
\caption{Geometrical illustration of the \eqref{eq:RAPG} algorithm.}
\label{figRAP}
\end{figure}

\medskip 

By following an argument similar to that of Theorem~\ref{highresolution}, the high resolution ODE of the algorithm \eqref{eq:RAPG} gives
\begin{equation}\label{HODE-Newton}
\ddot{Y}(t) + \frac{\alpha}{t}\dot{Y}(t) + \sqrt{s}\frac{d}{dt} \Big( T_s(Y(t)) \Big) + \left(1+\frac{\alpha \sqrt{s} }{2t} \right) T_s(Y(t))  = 0 ,
\end{equation}
where the term $\frac{d}{dt} \Big( T_s(Y(t)) \Big)$ is interpreted as the distributional derivative of the absolutely continuous function $t \mapsto T_s(Y(t))$. The ODE \eqref{HODE-Newton} is a Regularized Inertial Newton dynamic which has been recently studied in \cite{AAV1,AAV2} and \cite{AL1,AL2}. The existence and uniqueness of a strong solution to the Cauchy problem associated with \eqref{HODE-Newton} has been proved in \cite[Theorem~2.1]{AAV1}. It is based on the equivalent reformulation of \eqref{HODE-Newton} as a first-order system. 

Let us establish some fast convergence properties of \eqref{eq:RAPG} which can be deduced from the well established results concerning \eqref{eq:IPG}.
\begin{theorem}\label{thm:RAPG}
Suppose that \eqref{eq:assum} holds, $\alpha \geq 3$, and $sL \in ]0,1]$. Let $\seq{y_k}$ and $\seq{w_k}$ be the sequences generated by \eqref{eq:RAPG}. Then the following properties are satisfied:
\begin{equation*}
\theta(w_k)-\min_{\cH} \theta = \cO\left(\frac{1}{k^2}\right)  \,  \text{ \textit{as} } k\to +\infty;\end{equation*}
\noindent When  $\alpha>3$, 
\begin{equation*}
\theta(w_k)-\min_{\cH} \theta =o\left(\frac{1}{k^2}\right) \qandq \wlim y_k = \wlim w_k \in S .
\end{equation*}
\end{theorem}

\begin{proof}
As a key ingredient, we use the following analog of the gradient descent lemma for composite optimization, see \cite{BT}, \cite{CD}: for any $(x,y) \in \cH^2$
\begin{equation}
\theta (y- s T_s(y)) \leq \theta (x) + \left\langle T_s(y), y-x  \right\rangle -\frac{s}{2} \|T_s(y)\|^2.
\end{equation} 
Taking $x=w_{k-1}$ and $y= y_k$, we deduce that
\begin{equation}
\theta (w_k) \leq \theta (w_{k-1}) + \left\langle T_s(y_k), y_k-w_{k-1}  \right\rangle.
\end{equation} 
By Cauchy-Schwarz inequality, we get
\begin{equation}\label{est10-11-2021}
\theta (w_k) \leq \theta (w_{k-1}) + \| T_s(y_k)\| \| y_k-w_{k-1} \|.
\end{equation} 
By definition of \eqref{eq:RAPG}, and by using the link with \eqref{eq:IPG} (recall that $w_k=x_{k+1}$ where $\seq{x_k}$ are the iterates generated by \eqref{eq:IPG}), we have
\begin{eqnarray*}
\|T_s(y_k)\|&=&\frac{1}{s}\|w_k-y_k\|\\
&\leq& \frac{1}{s} \pa{\| w_k - w_{k-1}\|+\| y_k - w_{k-1}\|} \\
&\leq& \frac{1}{s} \pa{\| w_k - w_{k-1}\|+\| w_{k-1} - w_{k-2}\|} \\
&=& \frac{1}{s} \pa{\|x_{k+1}-x_k\| + \| x_{k} - x_{k-1}\|}=\cO \Big(\frac{1}{k}\Big) ,
\end{eqnarray*}
where the latter rate is known from \cite{ACPR,AP}. In a similar way,
\begin{equation}\label{eq:yktoxk}
\|y_k-w_{k-1}\| \leq  \| w_{k-1} - w_{k-2}\| = \|x_{k}-x_{k-1}\| = \cO \Big(\frac{1}{k}\Big).
\end{equation}
By plugging the above estimates into \eqref{est10-11-2021}, we obtain
\begin{equation*}
\theta (w_k) - \min_{\cH} \theta \leq \theta (x_k) - \min_{\cH} \theta + \frac{C}{k^2}= \cO \Big(\frac{1}{k^2}\Big).
\end{equation*} 
where the rate $\cO(1/k^2)$ on $\theta (x_k) - \min_{\cH}$ is known from \cite{BT}. 

\medskip

For $\alpha > 3$, we know from \cite{ACPR,AP} that $\theta (x_k) - \min_{\cH} = o\pa{\frac{1}{k^2}}$ and $\norm{x_k-x_{k-1}} = o\pa{\frac{1}{k}}$. We thus argue as above to obtain the claimed $o\pa{\frac{1}{k^2}}$ rate. In addition, we have from \eqref{eq:yktoxk} that $\| y_k -x_k \| \to 0$, \ie $y_k - x_k$  converges strongly to zero. Since the sequence $\seq{x_k}$ converges weakly when $\alpha >3$, see \cite{ACPR,AP}, it follows that the sequence $\seq{y_k}$ converges weakly to the same limit as $\seq{x_k}$.
\end{proof}

\section{The strongly convex case}
%
In this section, we briefly discuss the strongly convex case. Recall that $f: \cH \to \mathbb R$ is said to be $\mu$-strongly convex for some $\mu >0$ if   $f- \frac{\mu}{2}\| \cdot\|^2$ is convex. In this case, a proper tuning of the viscous damping coefficient in the dynamic \eqref{ODE001} of Polyak provides exponential convergence rate with optimal rate.
\begin{theorem}\label{strong-conv-thm}
Suppose that $f: \cH \to \R$ is a $\cC^1$ and $\mu$-strongly convex function for some $\mu >0$.
Let  $x(\cdot): [t_0, + \infty[ \to \cH$ be a solution trajectory of \eqref{ODE001} with $\gamma=2\sqrt{\mu}$, \ie
\begin{equation}\label{dyn-sc-a}
\ddot{x}(t) + 2\sqrt{\mu} \dot{x}(t)  + \nabla f (x(t)) = 0,
\end{equation}
with initial condition $(x(t_0),\dot{x}(t_0))$, $t_0 \geq 0$. Then, for all $t \geq t_0$ 
\[
f(x(t))-  \min_{\cH}f  \leq  C e^{-\sqrt{\mu}(t-t_0)} 
\] 
where $C \eqdef f(x(t_0))-  \min_{\cH}f  + \mu \norm{x(t_0)-x^\star}^2 + \| \dot{x}(t_0)\|^2 .$
\end{theorem}

\if
{

\begin{proof}
 Let $x^\star$ be the unique minimizer of $f$.
Define $\mathcal E : [t_0, +\infty[ \to \mathbb R^+ $  by
$$
\mathcal E (t):= f(x(t))-  \min_{\cH}f  + \frac{1}{2} \| \sqrt{\mu} (x(t) -x^\star) + \dot{x}(t)\|^2.
$$
Derivation of $\mathcal E (\cdot) $ gives 
$$
\frac{d}{dt}\mathcal E (t):=\langle  \nabla f (x(t)),  \dot{x}(t)       \rangle  + \langle \sqrt{\mu} (x(t) -x^\star) + \dot{x}(t), \sqrt{\mu}  \dot{x}(t)  +  \ddot{x}(t)        \rangle .
$$
Using the equation \eqref{dyn-sc-a}, we get
$$
\frac{d}{dt}\mathcal E (t)=\langle  \nabla f (x(t)),  \dot{x}(t)       \rangle  + \langle \sqrt{\mu} (x(t) -x^\star) + \dot{x}(t), -\sqrt{\mu}  \dot{x}(t)  - \nabla f (x(t))       \rangle .
$$
After developing and simplification, we obtain
$$
\frac{d}{dt}\mathcal E (t) + \sqrt{\mu}\langle  \nabla f (x(t)),  x(t) -x^\star       \rangle  + \mu\langle x(t) -x^\star , \dot{x}(t)\rangle + \sqrt{\mu} \| \dot{x}(t) \|^2   = 0.
$$
By the strong convexity of $f$ we have
$$
\langle  \nabla f (x(t)),  x(t) -x^\star       \rangle  \geq f(x(t))- f(x^\star) + \frac{\mu}{2} \| x(t) -x^\star \|^2 .
$$
By combining the two relations above we obtain 
$$
\frac{d}{dt}\mathcal E (t) + \sqrt{\mu}\left( f(x(t))- f(x^\star) + \frac{\mu}{2} \| x(t) -x^\star \|^2  + \sqrt{\mu} \langle x(t) -x^\star , \dot{x}(t)\rangle +  \| \dot{x}(t) \|^2   \right) \leq 0.
$$
Equivalently
$$
\frac{d}{dt}\mathcal E (t) + \sqrt{\mu}\left(\mathcal E (t)      
+ \frac{1}{2} \|  \dot{x} \|^2        \right) \leq 0.
$$
Integrating this differential inequality gives the claim. \qed
\end{proof}
}
\fi

According to the procedure described in section~\ref{continuous-discrete}, we consider three different discretizations of \eqref{dyn-sc-a} inspired by the inertial proximal algorithm, then the Nesterov method, and finally the Ravine method. Let $h>0$ be the temporal step size.

\subsection*{Inertial proximal algorithm} Implicit time discretization of \eqref{dyn-sc-a} gives
\[
\frac{ x_{k+1} - 2x_{k}+ x_{k-1} }{h^2} +   \sqrt{\mu} \frac{x_{k+1} - x_{k-1}}{h} + \nabla f( x_{k+1}) = 0.
\]
After multiplication by $s=h^2$, we obtain
\begin{equation}
(1+ h\sqrt{\mu}) (x_{k+1}-x_k)  + s \nabla f( x_{k+1}=   (1- h\sqrt{\mu})(x_{k} -x_{k-1}) , \label{basic-eq-2c}
\end{equation}
which gives
\begin{equation}\label{basic-eq-3c}
x_{k+1} = \prox_{\frac{s}{1+\sqrt{\mu s}}  f} \left(x_{k} +  \frac{ 1- \sqrt{\mu s}}{1+ \sqrt{\mu s} } (x_{k} -x_{k-1}) \right). 
\end{equation}
So, we obtain the inertial proximal algorithm 
\begin{equation}\label{eq:ipgsc}
\begin{cases}
y_k&=  x_{k} + \frac{ 1- \sqrt{\mu s}}{1+ \sqrt{\mu s} } (x_{k} -x_{k-1})  \\
x_{k+1} &= \prox_{\frac{s}{1+\sqrt{\mu s}}  f} \left( y_{k} \right)
\end{cases}
\end{equation}
\subsection*{Nesterov method} Replacing the proximal step by a gradient step in \eqref{eq:ipgsc}, we obtain 
\begin{eqnarray*}
\begin{cases}
y_k &=  x_{k} + \frac{ 1- \sqrt{\mu s}}{1+ \sqrt{\mu s} } (x_{k} -x_{k-1})  \\
x_{k+1} &=   y_{k} -\frac{s}{1+\sqrt{\mu s}}  \nabla f(y_k).
\end{cases}
\end{eqnarray*}
\subsection*{Ravine method} Interverting the role of the variables $x_k$ and $y_k$ we obtain the Ravine method
\begin{eqnarray*}
\begin{cases}
w_{k} &=   y_{k} -\displaystyle{\frac{s}{1+\sqrt{\mu s}}}  \nabla f(y_k)\\
y_{k+1} &=  w_{k} + \displaystyle{\frac{ 1- \sqrt{\mu s}}{1+ \sqrt{\mu s} } }(w_{k} -w_{k-1}) .
\end{cases}
\end{eqnarray*}
This scheme is closely related to the classical form of the Nesterov accelerated gradient method for strongly convex minimization. Note however that our approach makes appear a gradient step size  $\frac{s}{1+\sqrt{\mu s}}$ which is slightly different from the most usually used step size $s$. A detailed Lyapunov analysis for this scheme is an interesting subject that we leave for a future work.

\section{Comparison to related algorithms}

The following table gives a bird's eye view of the relationships between \eqref{eq:NAG}, \eqref{eq:RAG}.

\medskip

\begin{footnotesize}
\begin{center}
\begin{tabular}{|c||c||c|c|}
     \hline
& \multicolumn{3}{c|}{}  \\ 
 & \multicolumn{3}{c|}{Comparison of \eqref{eq:NAG} with \eqref{eq:RAG}}  \\ 
 & \multicolumn{3}{c|}{}  \\ 

\hline
\hline

 &&\multicolumn{2}{c|}{}\\ 
Algorithm & \eqref{eq:NAG} & \multicolumn{2}{c|}{\eqref{eq:RAG}} \\
\hline
&&\multicolumn{2}{c|}{}\\
Dual structure &  Extrapolation, then Gradient step& 
\multicolumn{2}{c|}{Gradient step, then Extrapolation}\\
&&\multicolumn{2}{c|}{}\\
\hline
&&\multicolumn{2}{c|}{}\\
Low resolution ODE &   \eqref{eq:AVD}  & \multicolumn{2}{c|}{\eqref{eq:AVD}}\\
&&\multicolumn{2}{c|}{}\\
\hline
&&\multicolumn{2}{c|}{}\\
High resolution ODE & Hessian driven damping (variable $x$)  & \multicolumn{2}{c|}{Hessian driven damping (variable $y$)}\\
&&\multicolumn{2}{c|}{}\\
\hline
&&\multicolumn{2}{c|}{}\\
Fast convergence of the gradients & Yes  & \multicolumn{2}{c|}{Yes}\\
&&\multicolumn{2}{c|}{}\\
\hline
&&\multicolumn{2}{c|}{}\\
Convergence rate, $\alpha >3$ &  $f (x_k)-\min_{\cH} f =   o \left(\displaystyle{\frac{1}{k^2} }\right)$  & 
\multicolumn{2}{c|}{$f (y_k)-\min_{\cH} f =   o \left(\displaystyle{\frac{1}{k^2} }\right)$ }\\
&&\multicolumn{2}{c|}{}\\
\hline
&&\multicolumn{2}{c|}{}\\
Convergence of iterates,  $\alpha >3$ & Yes & 
\multicolumn{2}{c|}{Yes}\\
&&\multicolumn{2}{c|}{}\\
\hline
&&\multicolumn{2}{c|}{}\\
Proximal version,  $\alpha >3$ & $f (x_k)-\min_{\cH} f =   o \left(\displaystyle{\frac{1}{k^2} }\right)$ & 
\multicolumn{2}{c|}{$f(\prox_{sf}(y_{k}))-\min_{\cH} f=o\left(\frac{1}{k^2}\right)$ }\\
&&\multicolumn{2}{c|}{}\\
 \hline
\end{tabular}
\end{center}
\end{footnotesize}

\subsection*{Comparison with \eqref{eq:IGAHD}}
Recall that the algorithm \eqref{eq:IGAHD}, introduced by the authors in \cite{ACFR}, is based on the dynamic  \eqref{ODE10}, a special case of \eqref{eq:DINAVD}, with damping parameters $\alpha \geq 3$ and $\beta \geq 0$.
This dynamic is essentially the same as the high resolution ODE \eqref{ODE101} associated to \eqref{eq:RAG} (the same holds for the ODE resp. \eqref{highresODENAG} of \eqref{eq:RAG}).
The two dynamics can be deduced from each other by a linear temporal reparameterization, which preserves their convergence properties. However, an important message to keep in mind here is that the algorithms \eqref{eq:RAG} and \eqref{eq:IGAHD} markedly differ in the type of temporal discretization used to obtain them. The algorithm \eqref{eq:RAG} is obtained by explicit discretization of \eqref{ODE101}, namely
\begin{equation}\label{dyn-Rav-discrete}
\frac{ y_{k+1} - 2y_{k}+ y_{k-1} }{h^2}  + \frac{\alpha}{(k+1-\alpha)h} \frac{y_{k+1} -y_k}{h} +\left( 1+ \frac{\alpha}{k+1-\alpha}\right) \nabla f (y_k)
+ \nabla f (y_k) -\nabla f (y_{k-1} ) =0.
\end{equation}
On the other hand \eqref{eq:IGAHD} is obtained by the time discretization of \eqref{ODE10}
\begin{eqnarray*}
\frac{1}{s}(x_{k+1} - 2x_{k}+ x_{k-1} ) +    \frac{\alpha}{ks}(x_{k} - x_{k-1}) +  \frac{\beta}{\sqrt{s}}(\nabla f( x_{k}) - \nabla f( x_{k-1})   ) + \frac{\beta}{k\sqrt{s}} \nabla f( x_{k-1})  +  
\nabla f( y_{k}) =0,
\end{eqnarray*}
with $y_k$ is an extrapolated point inspired by Nesterov's scheme. The convergence properties of \eqref{eq:IGAHD}, recalled in Theorem~\ref{thm:IGAHD} are similar to those \eqref{eq:RAG} and \eqref{eq:NAG} (see Theorem~\ref{thm:NAG}, Theorem~\ref{Rav-conv}, Theorem~\ref{pr.decay_E_k} and Theorem~\ref{thm:fastgradNAG}). In a nutshell, from a theoretical point of view, these methods behave similarly. Nevertheless, it was shown in \cite{ACFR} that, numerically, \eqref{eq:IGAHD} enjoys much less oscillations that \eqref{eq:NAG} (and thus \eqref{eq:RAG}). We conjecture that this is a consequence of the more subtle discretization underlying \eqref{eq:IGAHD}. So far, this lacks clear theoretical justification and we believe that it is a nice research program to undertake in the future.
%
%







\section{Conclusion, Perspective}

This work was intended to unveil the relationship between the Nesterov accelerated method and the Ravine method, which as been ignored for a long time and sometimes confused with the Nesterov. We have shed light on these connections through the perspective of dynamical systems. We believe that this work paves the way to many important questions that remain to be answered. Among them, we mention the following ones:
\begin{itemize}
\item Design better structure-preserving discretization schemes/algorithms for inertial systems, and understand their fundamental limits/performance.

\item For additively structured "smooth + nonsmooth" convex minimization problems, develop a Lyapunov analysis showing the fast convergence to zero of the operators (which correspond to the gradients for the Ravine accelerated gradient method).
 
\item Develop a Ravine accelerated method for linearly constrained optimization problems, which is based on the augmented Lagrangian approach, and the ADMM algorithm.

\item Study the introduction of perturbation, errors into the Ravine method, so as to prepare the stochastic versions of this algorithm.

\item Generalization and tuning of the extrapolation parameter.

\item The case of monotone inclusions.
\end{itemize}

\section*{Acknwoledgements}
We would like to thank N. Boumal who brought the reference \cite{Nesterov12} to our attention. 


\end{document}